\documentclass[a4paper, reqno]{amsart}
\usepackage[dvipsnames]{xcolor}
\usepackage{amsmath,amsthm,amssymb,relsize}
\usepackage[nobysame,alphabetic]{amsrefs}
\usepackage[margin=30mm]{geometry}

\usepackage{changes}

\usepackage{booktabs, multirow, adjustbox, array}

\usepackage{pifont}

\usepackage{enumerate}
\usepackage{tikz}
\usetikzlibrary{
  cd,
  calc,
  positioning,
  fit,
  arrows,
  decorations.pathreplacing,
  decorations.markings,
  shapes.geometric,
  backgrounds,
  bending
}
\usepackage{tikzsymbols}
\usepackage[arrow,matrix,curve,cmtip,ps]{xy}
\usepackage{pb-diagram}
\usepackage{pb-xy}

\usepackage{pinlabel}

\usepackage{subcaption}

\definecolor{darkblue}{rgb}{0,0,0.6}

\usepackage[breaklinks, pdftex, ocgcolorlinks,colorlinks=true, citecolor=darkblue, filecolor=darkblue, linkcolor=darkblue, urlcolor=darkblue]{hyperref}
\usepackage[capitalize,noabbrev]{cleveref}

\newtheorem{proposition}{Proposition}[section]
\newtheorem{theorem}[proposition]{Theorem}

\newtheorem{lemma}[proposition]{Lemma}
\newtheorem{indhyp}[proposition]{Inductive Hypothesis}
\newtheorem{conjecture}[proposition]{Conjecture}

\theoremstyle{definition}

\newtheorem{problem}[proposition]{Problem}

\newtheorem{convention}[proposition]{Convention}

\theoremstyle{remark}
\newtheorem{remark}[proposition]{Remark}

\newtheorem*{remark*}{Remark}

\newcommand{\RR}{\mathbb{R}}
\newcommand{\Z}{\mathbb{Z}}

\newcommand{\Id}{\operatorname{Id}}

\newcommand{\ol}{\overline}
\newcommand{\ul}{\underline}
\newcommand{\wt}{\widetilde}
\newcommand{\wh}{\widehat}
\newcommand{\sm}{\setminus}

\DeclareMathOperator{\Aut}{Aut}

\DeclareMathOperator{\Wh}{Wh}

\DeclareMathOperator{\pt}{pt}
\DeclareMathOperator{\pr}{pr}

\DeclareMathOperator{\TOP}{TOP}
\DeclareMathOperator{\OO}{O}
\DeclareMathOperator{\BTOP}{BTOP}
\DeclareMathOperator{\BO}{BO}

\DeclareMathOperator{\PL}{PL}
\DeclareMathOperator{\Homeo}{Homeo}

\begin{document}
\title{Pseudo-isotopies of simply connected 4-manifolds}

 \author[D.~Gabai]{David Gabai}
\address{Department of Mathematics, Princeton University, NJ 08540, USA}
 \email{gabai@math.princeton.edu}

\author[D.~T.~Gay]{David T. Gay}
\address{Department of Mathematics, University of Georgia, Athens, GA 30602, USA}
 \email{dgay@uga.edu}

 \author[D.~Hartman]{Daniel Hartman}
\address{Max Planck Institute for Mathematics in Bonn, Vivatsgasse 7, 53111 Bonn, Germany}
 \email{daniel.hhartman@gmail.com}

 \author[V.~Krushkal]{Vyacheslav Krushkal}
 \address{Department of Mathematics, University of Virginia, Charlottesville, VA 22903, USA
 }
 \email{krushkal@virginia.edu}

 \author[M.~Powell]{Mark Powell}
 \address{School of Mathematics and Statistics, University of Glasgow, United Kingdom}
 \email{mark.powell@glasgow.ac.uk}

\def\subjclassname{\textup{2020} Mathematics Subject Classification}
\expandafter\let\csname subjclassname@1991\endcsname=\subjclassname

\begin{abstract}
Perron and Quinn gave independent proofs in 1986 that every topological pseudo-isotopy of a simply-connected, compact topological 4-manifold is isotopic to the identity.  Another result of Quinn is that every smooth pseudo-isotopy of a simply-connected, compact, smooth 4-manifold is smoothly stably isotopic to the identity. From this he deduced that $\pi_4(\TOP(4)/\OO(4)) =0$.
A replacement criterion is used at a key juncture in Quinn's proofs, but the justification given for it is incorrect.  We provide different arguments that bypass the replacement criterion, thus completing Quinn's proofs of both the topological and the stable smooth pseudo-isotopy theorems. We discuss the replacement criterion and state it as an open problem.
\end{abstract}
\maketitle

\section{Introduction}

Let $M$ be a compact $d$-manifold.
Let $\sqsubset$ denote $(M \times \{0\}) \cup (\partial M \times [0,1]) \subseteq M \times [0,1]$.
A \emph{pseudo-isotopy of $M$} is a homeomorphism $F \colon M \times [0,1] \to M \times [0,1]$ such that $F|_{\sqsubset} = \Id_{\sqsubset}$.   The homeomorphism $f \colon M \to M$ obtained from restricting $F$ to $M \times \{1\}$ is said to be \emph{pseudo-isotopic to the identity}.  If $F$ is level-preserving, i.e.\ if $F|_{M \times \{t\}}$ is a homeomorphism from $M \times \{t\}$ to itself for all $t \in [0,1]$, then $F$ is the trace of an isotopy; for brevity in this case we refer to $F$ simply as an isotopy.  Observe that if $F$ is an isotopy then $F$ is isotopic rel.\ $\sqsubset$ to the identity map of $M \times [0,1]$, and in particular $f$ is isotopic to the identity map of~$M$.  There are directly analogous definitions in the smooth and $\PL$ categories.

For $M$ smooth and simply-connected, $F$ a diffeomorphism, and $d \geq 5$, Cerf~\cite{Cerf} proved that $F$ is smoothly isotopic to the identity, which in particular implies that $f$ is smoothly isotopic to the identity. For $d \geq 6$, Cerf's method was extended by Hatcher--Wagoner~\cite{HW} and Igusa~\cite{Igusa-what-happens} to a two-stage obstruction theory deciding in the non-simply-connected setting whether $F$ is smoothly isotopic to the identity.  Cerf and Hatcher--Wagoner's results were extended to the topological and $\PL$ categories by Pedersen~\cite{BLR}*{Appendix~2}, making use of work of Hudson~\cite{Hudson} and~\cite{Pedersen-top-conc}.
The case $d=5$ in the non-simply connected case is open at the time of writing, in all categories.

This article concerns the extension of Cerf's theorem to dimension four, in the topological category and in the smooth category after stabilising $M \times [0,1]$ with copies of $S^2 \times S^2 \times [0,1]$. From now on we set $d=4$.

\subsection{Stable isotopy}\label{subsec:intro-smooth-stable}

First we discuss the smooth stable version.  We recall the definition of a stable isotopy.
Assume that $M$ is a compact, smooth, simply-connected 4-manifold, and that $F \colon M \times [0,1] \to M \times [0,1]$ is a smooth pseudo-isotopy.
After an isotopy, we may assume that there is a 4-ball $D^4 \subseteq M$ such that $F|_{D^4 \times [0,1]}$ is the identity. We can then connect sum
$M \times [0,1]$ with $(\#^k S^2 \times S^2) \times [0,1]$ along $D^4 \times [0,1]$; let $N_k$ denote the result. Consider $F_k\colon N_k\rightarrow N_k$, extending $F|_{(M \times [0,1] )\sm (\mathring{D}^4 \times [0,1])}$ by the identity on $((\#^k S^2 \times S^2) \times [0,1]) \sm (\mathring{D}^4 \times [0,1])$.
If there exists a $k$ such that $F_k$ is smoothly isotopic rel.\ $\sqsubset$ to the identity, then we say that the pseudo-isotopy $F$ is \emph{stably isotopic to the identity}.

Here is Quinn's $4$-dimensional smooth stable pseudo-isotopy theorem~\cite{Quinn:isotopy}*{Theorem~1.4}.

\begin{theorem}\label{thm:PI-stable-I}
Let $M$ be a compact, smooth, simply-connected 4-manifold and let $F \colon M \times [0,1] \to M \times [0,1]$ be a smooth pseudo-isotopy. Then $F$ is stably isotopic  to the identity.
\end{theorem}

In Section \ref{section:DRC-problem} we show that Quinn's proof of one of the steps in his argument, the \emph{Replacement Criterion} in \cite{Quinn:isotopy}*{Section~4.5}, does not work.
We do not know whether the replacement criterion holds, and this is an interesting open question; we state it in \cref{problem:DRC}.

The first main goal of this article is to fix Quinn's proof of \cref{thm:PI-stable-I}. We will use the majority of Quinn's argument, but we modify it to replace his use of the replacement criterion. Our modification uses a method called \emph{factorisation}, which first appeared in~\cite{Gabai-22}*{Lemma~3.15}, to split a pseudo-isotopy in two; see below \cref{remark:separate-proof}. One of the two resulting pseudo-isotopies can be stably isotoped to the identity using Quinn's methods. We present a new argument, a novel application of Quinn's sum square move, to resolve the other pseudo-isotopy. We exploit extra $S^2 \times S^2$ summands to find geometrically dual, framed embedded spheres to certain surfaces, when they are required.

\begin{remark}
Stabilisation is necessary in \cref{thm:PI-stable-I}.
In dimension 4, smooth pseudo-isotopy does not imply smooth isotopy for 1-connected 4-manifolds, as shown first by Ruberman~\cite{ruberman-isotopy}, with later examples constructed by Kronheimer-Mrowka~\cite{Kronheimer-Mrowka-K3},  Baraglia-Konno~\cite{Baraglia-Konno}, Lin~\cite{Lin-dehn-twist-stabn}, and Konno-Mallick-Taniguchi~\cite{Konno-Mallick-Taniguchi}, among others.
\end{remark}

Quinn deduced the following result from \cref{thm:PI-stable-I}, stated as~\cite{Quinn:isotopy}*{Theorem~1.2}. Here $\TOP(4)$ denotes the topological group of homeomorphisms of $\RR^4$ that fix the origin, and $\TOP(4)/\OO(4)$ denotes the homotopy fibre of the canonical map  $\BO(4) \to \BTOP(4)$.

\begin{theorem}\label{thm:top4-o4}
$\pi_4(\TOP(4)/\OO(4)) =0$.
\end{theorem}

The space $\TOP(4)/\OO(4)$ is an important universal space in smoothing theory.
For example, \cref{thm:top4-o4} was the final step in showing \cite{FQ}*{Theorem~8.7A}, that $\TOP(4)/\OO(4) \to \TOP/\OO$ is 5-connected. In combination with Lees and Lashof's immersion theory~\cites{Lees,Lashof-70-I,Lashof-70-II,Lashof-71}, this implies \cite{FQ}*{Theorem~8.7B}, which states that for $M$ a noncompact, connected 4-manifold, concordance classes of smooth structures correspond bijectively with the cohomology group $H^3(M,\partial M;\Z/2)$.

\begin{remark}
There is now an alternative proof due to Gabai~\cite{Gabai-22}*{Theorem~2.5}, of the fact that smoothly pseudo-isotopic diffeomorphisms of a compact 1-connected 4-manifold are stably isotopic. However Gabai's proof does not trivialise the given pseudo-isotopy, so does not recover \cref{thm:PI-stable-I}.
Thus Gabai's theorem~\cite{Gabai-22}*{Theorem~2.5} cannot be applied to prove~\cref{thm:top4-o4}.

We note that Perron's article~\cite{Perron}, which we will discuss more below, did not address \cref{thm:PI-stable-I}, and it is not clear how to approach it using Perron's method, due to his use of the Alexander trick.
Thus as far as we know \cref{thm:top4-o4} cannot be deduced from Perron's work.

The only proof known to us of \cref{thm:top4-o4} makes use of our corrected proof of \cref{thm:PI-stable-I}.
\end{remark}

\subsection{Topological isotopy}
Now we discuss the topological 4-dimensional analogue of Cerf's theorem.

\begin{theorem}\label{thm:main-PI-I}
  Let $M$ be a compact, topological, 1-connected 4-manifold and let $F \colon M \times [0,1] \to M \times [0,1]$ be a pseudo-isotopy. Then $F$ is topologically isotopic rel.\ $\sqsubset$ to the identity $\Id_{M \times [0,1]}$.
\end{theorem}

The first proof of \cref{thm:main-PI-I}, for a class of $4$-manifolds discussed next, was given by Perron~\cite{Perron}. Perron's main result states the theorem for smooth, compact 4-manifolds without 1-handles. In \cite{Perron}*{Section~7}, Perron deduced the statement for closed, 1-connected, topological 4-manifolds, using an argument he attributes to Siebenmann. In \cref{sec:Perron} we will explain how this deduction can be extended further to compact manifolds with boundary using work of Boyer~\cites{Boyer86,Boyer93} to deduce all cases of \cref{thm:main-PI-I}.
This extension is necessary  because Casson \cite{Kirby-problems}*{Remarks following Problem 4.18} showed there are compact, contractible, smooth 4-manifolds with boundary that do not admit a handle decomposition with no 1-handles.

Quinn~\cite{Quinn:isotopy}*{Theorem~1.4} gave an independent proof of \cref{thm:main-PI-I}, for all  compact, topological, 1-connected 4-manifolds.
Quinn's work may be thought of as a natural extension of Cerf's approach to dimension $4$. A (presently unknown) extension to non-simply connected $4$-manifolds would likely involve a combination of the higher-dimensional Hatcher--Wagoner theory for non-trivial fundamental groups and a suitable extension of Quinn's approach.

The second main goal of this article is to fix Quinn's proof of \cref{thm:main-PI-I}, which also relied on his replacement criterion~\cite{Quinn:isotopy}*{Section~4.5}.
Thus, given the extension of Perron's proof to manifolds with boundary, discussed above, and our completion of Quinn's argument, there are two independent proofs of \cref{thm:main-PI-I} in full generality.
As with the fix of \cref{thm:PI-stable-I}, we will retain the majority of Quinn's argument, however we will again split the given pseudo-isotopy into two using factorisation. One of the resulting pseudo-isotopies can be resolved using Quinn's argument from \cite{Quinn:isotopy}*{Section~4.6}. For the second we arrange for a judicious  application of the Alexander trick.

\begin{remark}
As discussed above, there is a well established obstruction theory for pseudo-isotopies of non-simply-connected manifolds in dimensions $\geq 6$, in both topological and smooth categories. An analogue of such a theory is not presently known for non-simply-connected  $4$-manifolds.
However, it follows from the work of Budney-Gabai~\cite{Budney-Gabai-2}
that the simply-connected hypothesis cannot be removed in \cref{thm:main-PI-I}.
\end{remark}

\subsection{Classification of homeomorphisms up to isotopy}

An important application of \cref{thm:main-PI-I} is to classify homeomorphisms of 1-connected 4-manifolds up to isotopy. To achieve this one also needs a classification of homeomorphisms up to pseudo-isotopy.

The classification of diffeomorphisms up to smooth pseudo-isotopy was done using modified surgery by Kreck~\cite{Kreck-isotopy-classes} for closed, smooth, 1-connected 4-manifolds.
Quinn gave the classification for homeomorphisms of closed, topological, 1-connected 4-manifolds in \cite{Quinn:isotopy}, with a correction of the analysis of the normal invariants given by Cochran-Habegger~\cite{Cochran-Habegger}.  It is also straightforward to adapt Kreck's argument to the topological category.

The outcome \cite{Quinn:isotopy}*{Theorem~1.1}  of the classification of homeomorphisms up to pseudo-isotopy, combined with \cref{thm:main-PI-I}, is that the natural map
\[\pi_0\Homeo^+(M) \xrightarrow{\cong} \Aut(H_2(M),\lambda_M),\]
sending an isotopy classes of orientation-preserving homeomorphisms of $M$ to the induced isometry of the intersection pairing $\lambda_M \colon H_2(M) \times H_2(M) \to \Z$, is an isomorphism.
Surjectivity is due to Freedman~\cite{F}.

For homeomorphisms of compact, 1-connected, topological manifolds, with possibly nonempty boundary, the classification up to pseudo-isotopy, and hence up to isotopy by \cref{thm:main-PI-I}, was completed by Orson-Powell~\cite{Orson-Powell}.
The corresponding surjectivity result in the case of nonempty boundary is due to Boyer~\cite{Boyer86}.

\subsection*{Organisation}

We begin by explaining our remedy for the proof of \cref{thm:PI-stable-I}.
In \cref{sec:setting-up} we recall the start of Quinn's proof in the smooth setting, and state the key propositions from his paper that we will use.
\cref{sec:sum-square-move} recalls two key geometric constructions: Whitney spheres and the sum square move.
\cref{section:fixing-the-proof} contains the details of our fix for the proof of \cref{thm:PI-stable-I}.

\cref{sec:Perron} proves a key lemma on a decomposition of compact 1-connected topological 4-manifolds. It has two aims. It gives a key input for Quinn's topological approach, and it enables us to explain how Perron's method can be applied to all compact 1-connected 4-manifolds, and not just those that are either closed or admit a handle decomposition without 1-handles.

In \cref{section:fixing-the-proof-topologically} we adapt our fix from \cref{section:fixing-the-proof} to the topological case.
Finally, in \cref{section:DRC-problem} we give more details on the problem with Quinn's proof of the Replacement Criterion, and pose the question of whether the replacement lemma holds.

\subsection*{Acknowledgements}
We are grateful to Frank Quinn for graciously discussing his work with us, to Daniel Galvin for pointing out a subtlety with the application of smoothing theory at the start of the proof of \cref{prop:goal}, and to the referee for several useful suggestions.

David Gabai was partially supported by NSF grant DMS-2304841.
David T.\ Gay and Daniel Hartman were partially supported by NSF grant DMS-2005554. David T.\ Gay was also supported in part by Simons Foundation grant MP-TSM-00002714.
Vyacheslav Krushkal was supported in part by NSF grants DMS-2105467 and DMS-2405044.
Mark Powell was partially supported by EPSRC New Investigator grant EP/T028335/2 and EPSRC New Horizons grant EP/V04821X/2.

\section{Morse functions, nested eye Cerf graphics, and the start of Quinn's proof }\label{sec:setting-up}

In this section we summarise the start of Quinn's proof assuming the input of a smooth pseudo-isotopy.   We explain the position one is in prior to Quinn's use of the replacement criterion, and to set up notation.
After the replacement criterion, Quinn's argument for completing the stable smooth proof is also valid, and we make use of this too, in the form of \cref{thm:quinn-alg-ints-implies-close-eye-smooth-stable-case}  below.

Throughout this section and \cref{section:fixing-the-proof} we work in the smooth category.  In \cref{sec:Perron,section:fixing-the-proof-topologically} we will explain how to adapt the arguments to the topological category, unstably.

A smooth pseudo-isotopy $F \colon M \times [0,1] \to M \times [0,1]$ of a smooth compact 1-connected 4-manifold can be translated into a 1-parameter family of generalised Morse functions $G_t \colon M \times [0,1] \to [0,1]$, where $G_0 = \pr_2$ and $G_1 = \pr_2 \circ F$. Here $\pr_2$ denotes projection onto the second factor, namely~$[0,1]$. Both $G_0$ and $G_1$ have no critical points.    The family $G_t$ is also accompanied by a 1-parameter family of gradient-like vector fields $\xi_t$ subordinate to $G_t$.

Given a pair $(G_t,\xi_t)$, one can in turn recover a pseudo-isotopy $F \colon M\times [0,1] \to M \times [0,1]$.
The proof of the fact that pseudo-isotopy implies isotopy consists of \emph{deformations} of these 1-parameter families, i.e.\ a 1-parameter family of 1-parameter families. A deformation corresponds to an isotopy of the pseudo-isotopy $F$. The aim is to deform $(G_t,\xi_t)$ until there are no critical points of $G_t$, for all $t \in [0,1]$. Such a 1-parameter family corresponds to a pseudo-isotopy $F$ that is an isotopy.

Using $\Wh_2(\{1\})=\{1\}$, by the work of Hatcher--Wagoner~\cite{HW}*{Chapter~VI,~Proposition~3} one can arrange by a deformation of the family that the Cerf graphic associated to $G_t$ consists of finitely many nested eyes -- pairs of index 2 and 3 critical points that are born near time $t_b$, and die at time $t_d > t_b$, with no rearrangements. More precisely, we let $t_b$ be the time of the final birth, and let $t_d$ be the time of the first death.  See Figure \ref{figure:Cerf graphic}.
According to Hatcher--Wagoner~\cite{HW}*{Chapter~VI,~Proposition~3}, we may assume that there are no handle slides, no critical value crossings, and that the births and deaths are independent. In particular, there are no trajectories between any pair of index 3 critical points, and similarly for any pair of index 2 critical points. Moreover at the moment of each birth and shortly thereafter, the index $2$ and $3$ critical points that appear do not have trajectories to any other critical points, other than the unique trajectory between them. The same holds at each death, and shortly before.

\begin{figure}[ht]
\centering
\includegraphics[height=4.2cm]{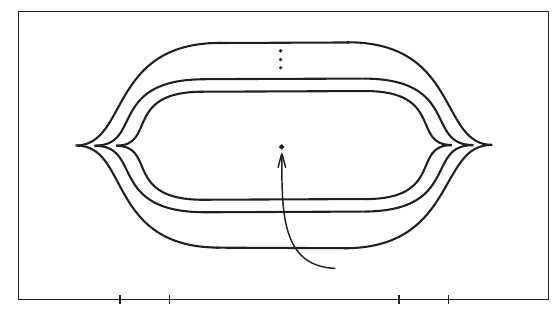}
{\small
\put(-170,-7){$t_b$}
\put(-150,-7){$t_f$}
\put(-62,-7){$t_w$}
\put(-44,-7){$t_d$}}
{\scriptsize
\put(-83,13){\rm middle middle level}
}
\caption{A family of nested eyes.}
\label{figure:Cerf graphic}
\end{figure}
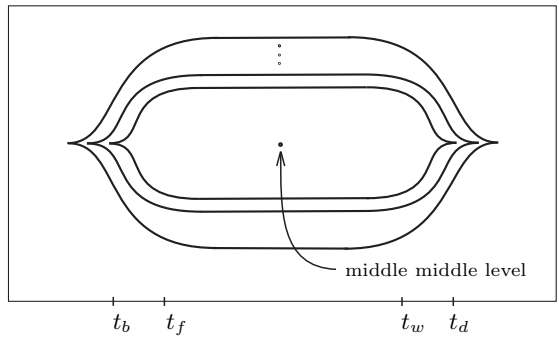

If we can `close' the innermost eye in the family of nested eyes, one at a time, then we complete the proof of the smooth stable isotopy theorem. We formalise this in the following inductive hypothesis.

\begin{indhyp}[$\mathbf{E(n)}$]
\label{ind-hyp}
    Let $M$ be a smooth 1-connected 4-manifold and let $(G_t,\xi_t)$ be as above, a 1-parameter family of generalised Morse functions and gradient like vector fields on $M \times [0,1]$. Suppose that the associated Cerf graphic consists of $n$ nested eyes with cancelling pairs of index 2 and 3 critical points, with no handle slides.  Then after stabilising $M \times [0,1]$ with $(\#^p S^2 \times S^2) \times [0,1]$, for some $p$,  there is a deformation of $(G_t,\xi_t)$ to a 1-parameter family without critical points.
\end{indhyp}

We will prove the base case $E(1)$ and we will deduce the case $E(n)$ from $E(n-1)$.
The core of Quinn's proof starts with a nested eye family, and works on the innermost eye.
We set up notation and recall the salient points of Quinn's argument for removing the innermost eye.

Let us suppose we have a family of $n$ nested eyes, for some $n \geq 1$.
In the level sets $G_{t}^{-1}(1/2)$, for $t_b < t< t_d$, we have a copy of $M \#_{i=1}^n S^2 \times S^2$.
For each $t$ we have $n$ spheres $\{A_i^t\}_{i=1}^n$, which are the ascending spheres of the index 2 critical points, and $n$ spheres $\{B_i^t\}_{i=1}^n$, which are the descending spheres of the index 3 critical points.  Shortly after the last birth time $t_b$, the sphere $A_i^t$ is the sphere $S^2 \times \{\pt\}$ in the $i$th $S^2 \times S^2$ summand. Similarly the sphere $B_i^t$ is the sphere $\{\pt\} \times S^2$ in the $i$th $S^2 \times S^2$ summand.
The $A_i^t$ and the $B_i^t$ are enumerated so that $A_1$ and $B_1$ correspond to the innermost eye, and then moving outwards, as shown in e.g.\ \cref{figure:From-VW-to-VW’W’W}.  Let $A^t := \cup_{i=1}^n A_i^t$ and let $B^t := \cup_{i=1}^n B_i^t$.

One may specify an identification of $G_{t}^{-1}(1/2)$ with $M \#_{i=1}^n S^2 \times S^2$ for all $t_b < t< t_d$. This is implicit in \cite{Quinn:isotopy} and described in detail in \cite{KMPW}*{Construction 3.1}; see also \cite{gay2021diffeomorphisms}*{Proof of Theorem 9}. Then it makes sense to discuss subsets of $G_t^{-1}(1/2)$, for $t_b < t< t_d$ as being in the fixed manifold $M \#_{i=1}^n S^2 \times S^2$. Following \cite{Quinn:isotopy}*{p.~353}, we
may consider that $A^t$ stays fixed as $t$ varies in $t_b <t <t_d$, $B^t$ moves by an isotopy, and $A^t \cup B^t$ undergoes a regular homotopy.
We may also assume that there are times $t_f$ and $t_w$, with $t_b < t_f < 1/2 < t_w < t_d$, such that at time $t_f$ a collection of finger moves  occur, and at time $t_w$ a collection of Whitney moves occur.  During these moves,  intersections of $B_j^t$ with $A_i^t$ are created or removed, respectively.   We consider the \emph{middle-middle level}, $G_{1/2}^{-1}(1/2)$.  We can let the finger and Whitney discs evolve so that we see both simultaneously in the middle-middle level.  We write $A_i := A_i^{1/2}$ and $B_j := B^{1/2}_j$; that is, in the middle-middle level we drop the $t$ from the notation.

In the middle-middle level, we call the finger move discs for $A_i$, $B_j$ intersections $\{V^{ij}_k\}$, and write $V^{ij} := \cup_k V^{ij}_k$. Analogously we call the Whitney discs $\{W^{ij}_{\ell}\}$, and write $W^{ij} := \cup_{\ell} W^{ij}_{\ell}$.  We also write $V := \cup_{i,j} V^{ij}$ and $W := \cup_{i,j} W^{ij}$.

A Cerf family $(G_t,\xi_t)$, up to deformation, determines and is determined by the data $(A,B,V,W)$ in $M \#_{i=1}^n S^2 \times S^2$. This is implicit in \cite{Quinn:isotopy}; a detailed discussion is given in \cite{KMPW}*{Section 3.1}. One of Quinn's insights was to describe  modifications of this data that correspond to deformations of the family $(G_t,\xi_t)$. The idea of his proof is to start with $(G_t,\xi_t)$, consider the corresponding $(A,B,V,W)$, suitably modify them, and then deduce that $(G_t,\xi_t)$ can be deformed to a Cerf family without critical points.

Working on the innermost eye, Quinn showed in \cite{Quinn:isotopy}*{Sections~4.2--4.4} that we can assume, after a deformation, that both \[\big(\bigcup_k \partial V^{11}_k \cup \partial W^{11}_k\big) \cap A_1\] and
\[\big(\bigcup_k \partial V^{11}_k \cup \partial W^{11}_k\big) \cap B_1\] are arcs in $A_1$ and $B_1$ respectively.  We will refer to this as \emph{Quinn's arc condition.}
See \cref{figure:Wh-discs-with-arc-boundaries}.

\begin{convention}
We will assume that the finger and Whitney discs, $\{ V^{11}_k\}$ and $\{ W^{11}_l\}$, for the intersections in the innermost eye, are indexed according to their order in the arc, as indicated in \cref{figure:Wh-discs-with-arc-boundaries}.
\end{convention}

\begin{figure}
\centering
\includegraphics[height=2.7cm]{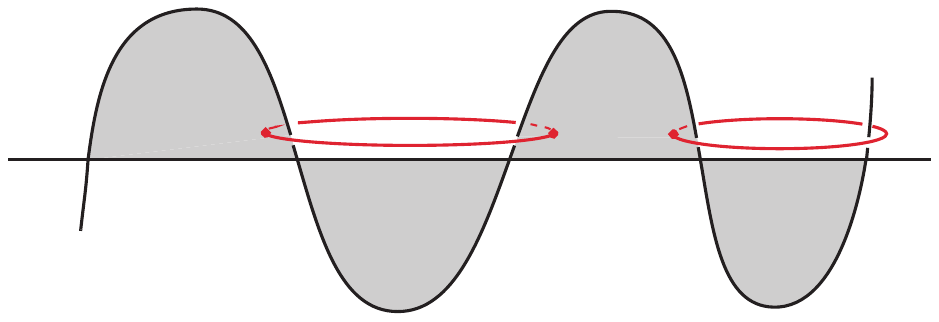}
{\scriptsize
\put(-225,42){$A_1$}
\put(-203, 20){$B_1$}
\put(-185, 55){$W^{11}_1$}
\put(-85,55){$W^{11}_2$}
\put(-134,20){$V^{11}_1$}
\put(-45,20){$V^{11}_2$}
\put(-41,56){{\color{red}$S_{V^{11}_2}$}}
\put(-135,56){{\color{red}$S_{V^{11}_1}$}}
}
\caption{Finger and Whitney discs with boundaries forming an arc in $A_1$ and in $B_1$. The Whitney spheres from \cref{subsection:whitney spheres} corresponding to the finger move discs are also shown.}
\label{figure:Wh-discs-with-arc-boundaries}
\end{figure}
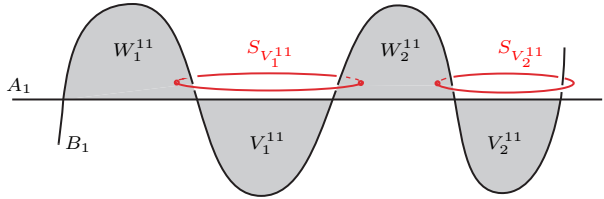

For a disc $D$, let $\mathring D$ denote its interior.
The following statement summarizes the result of  \cite{Quinn:isotopy}*{Section~4.6}.

\begin{theorem}[Quinn]\label{thm:quinn-alg-ints-implies-close-eye-smooth-stable-case}
    Suppose that $M$ is smooth, compact, and 1-connected, and suppose that $F \colon M \times [0,1] \xrightarrow{\cong} M \times [0,1]$ is a smooth pseudo-isotopy. Consider an associated 1-parameter family $(G_t,\xi_t)$, and suppose that the associated Cerf graphic consists of $n$ nested eyes with cancelling pairs of index 2 and 3 critical points, and no handle slides. Consider the data of spheres and finger/Whitney discs $(A,B,V,W)$ in the middle-middle level.

    If Quinn's arc condition is satisfied for $(A,B,V,W)$, and if the algebraic intersection numbers $\mathring{V}^{11}_k \cdot \mathring{W}^{11}_{\ell}$ vanish for all $k,\ell$, then after stabilisation of the pseudo-isotopy with $(\#^p S^2 \times S^2) \times [0,1]$, for some $p$, there is a smooth deformation of the family to one with $(n-1)$ nested eyes.
\end{theorem}

Quinn used the problematic replacement criterion to arrange for the algebraic intersection condition in \cref{thm:quinn-alg-ints-implies-close-eye-smooth-stable-case}  to hold. Since we cannot appeal to the replacement criterion, we must provide an alternative argument.

\section{Whitney spheres and the sum square move}\label{sec:sum-square-move}

Here we take a short digression to recall two important constructions that we shall need in \cref{section:fixing-the-proof}.  In \cref{subsection:whitney spheres} we recall the construction of Whitney spheres. In \cref{subsection:sum-square} we recall Quinn's sum square move.

\subsection{Whitney spheres}\label{subsection:whitney spheres}

    We recall a construction of the \emph{Whitney sphere} $S_{V_k^{ij}}$ associated with a finger move disc $V_k^{ij}$. These spheres have appeared in different guises in the literature, cf.~ \cite{Quinn:isotopy}*{Section 4.3}, \cite{FQ}*{Section 3.1, Ex.~(2)}, \cite{COP20}*{Section~4.2}, \cite{ST-2019}*{Section 2}. We use the terminology
    and the description from \cite{ST-2019}.

    The description is given in ${\mathbb R}^3\times {\mathbb R}$ where $A_i$ and the finger disc $V_k^{ij}$ are in ${\mathbb R}^3\times \{0\}$, and $B_j$ is represented as (arc $\subseteq {\mathbb R}^3$)$\times [-1,1]$, Figure \ref{figure:Whitney sphere}.
    The Whitney sphere is drawn red, and it consists of two discs,  $D_{i}\subseteq {\mathbb R}^3\times\{ i\}$, $i=-1,1$, joined by an annulus (circle $\subseteq {\mathbb R}^3$)$\times [-1,1]$. Each $D_i$ is constructed using two copies of the finger move disc $V_k^{ij}$, so overall the Whitney sphere contains four pushed-off copies of $V_k^{ij}$.

    The Whitney sphere is framed and embedded and can be assumed to lie in an arbitrarily small neighbourhood of $V_k^{ij}$.

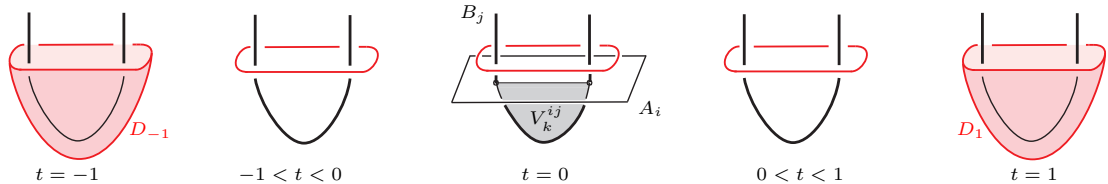
\begin{figure}[h]
\centering
\includegraphics[height=2.15cm]{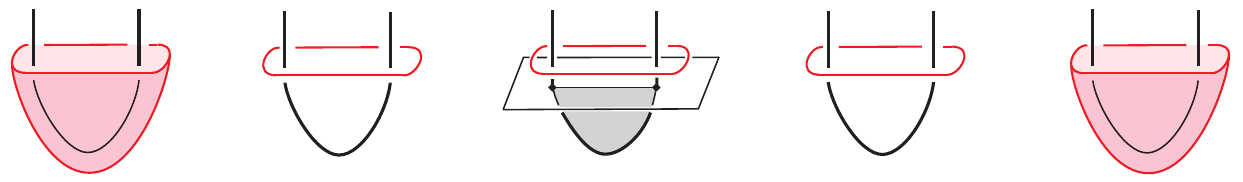}
{\scriptsize
\put(-220, 16){$V_k^{ij}$}
\put(-182, 20){ $A_i$}
\put(-246,56){$B_j$}
\put(-405, -5){$t=-1$}
\put(-329, -5){$-1<t<0$}
\put(-223, -5){$t=0$}
\put(-135, -5){$0<t<1$}
\put(-40, -5){$t=1$}
\put(-370,12){{\color{red}$D_{-1}$ }}
\put(-60,12){{\color{red}$D_{1}$ }}
}
\caption{A description of the Whitney sphere $S_{V_k^{ij}}$ in $\mathbb{R}^3\times \mathbb{R}$.}
\label{figure:Whitney sphere}
\end{figure}

\subsection{The sum square move}\label{subsection:sum-square}

We recall Quinn's sum square move \cite{Quinn:isotopy}*{Section~4.2}, which can be used to modify finger or Whitney disc configurations by a deformation of the pseudo-isotopy.
The data for the move is a framed embedded square~$S$ in the middle-middle level, with interior disjoint from the spheres $A$ and~$B$, and from the discs~$V$. The square has two edges on the $V$ discs, denoted $V_0$ and $V_i$ (anticipating the proof below) in Figure \ref{figure:sum square}, one edge on $A$, and one on $B$. New $V$ discs are obtained by cutting $V_0, V_i$
along the boundary edges of the sum square $S$, and gluing in two parallel copies of $S$. The effect of the move on the boundaries of the discs, on $A$ and~$B$ spheres, is illustrated in Figure \ref{figure:sum square boundaries}.

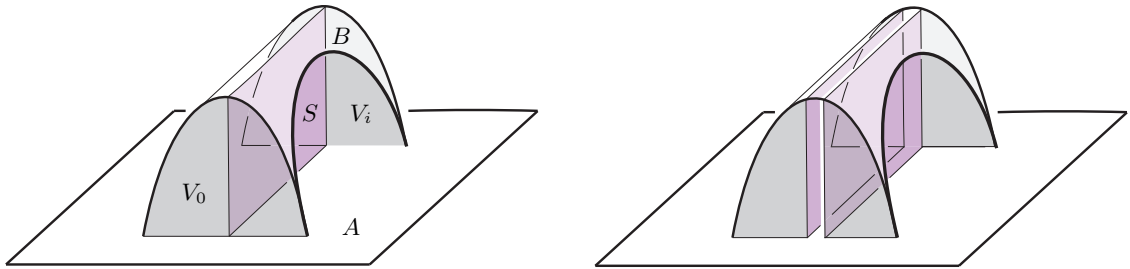
\begin{figure}[h]
\centering
\includegraphics[height=3.8cm]{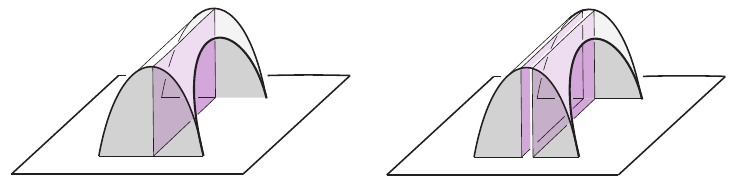}
{\small
\put(-304,89){$B$}
\put(-300,17){$A$}
\put(-360,30){$V_0$}
\put(-297,60){$V_i$}
\put(-315,60){$S$}
}
\caption{The sum square move}
\label{figure:sum square}
\end{figure}

Figure \ref{figure:sum square} is a $3$-dimensional model for the sum square. Here $A, V_0, V_i$, and a neighbourhood of the arc of $\partial S$ in $B$  are pictured in the ``present'', ${\mathbb R}^3\times\{ 0\}\subseteq {\mathbb R}^3\times {\mathbb R}$. The rest of $B$ extends into the past and the future. The framing of $S$ along its boundary is determined in the $3$-dimensional model by a non-vanishing vector field on $\partial S$ which is normal to $S$ and tangent to $A, B$ and the $V$ discs; this framing has to admit an extension over $S$ for the move to give rise to embedded $V$ discs.

A justification is given in \cite{Quinn:isotopy}*{Section~4.2} for why the sum square move gives a deformation of the pseudo-isotopy. As usual in the subject, given the boundary data $\partial S$, the challenge is to find a framed, embedded $S$ disjoint from other given surfaces. The proof in \cref{section:fixing-the-proof} below explains how to achieve this in relevant situations in the stable setting.
Note that if $S \cap W = \emptyset$, then the~$V$ discs do not acquire new intersections with $W$ as a result of this move.

\section{The proof of the inductive step}\label{section:fixing-the-proof}

We continue with the notation and setup from \cref{sec:setting-up}, and describe our correction to Quinn's proof of~\cref{thm:PI-stable-I}.

Suppose that there are $m$ finger discs $\{V^{11}_k\}_{k=1}^m$. Then there are also $m$ Whitney discs $\{W^{11}_{\ell}\}_{\ell =1}^m$ for $A_1$--$B_1$ intersections.  We fix an orientation on each of these discs.
Define \[T_{\ell} := \#_{k=\ell}^m S_{V^{11}_k},\] where the connected sum is formed by tubing the Whitney spheres $S_{V^{11}_k}$ together along arcs in the $W^{11}_{k}$ that are parallel to one of the boundary arcs of $W^{11}_{k}$, for $k >\ell$. (The tubing could be done along any arcs, not necessarily arcs in the Whitney discs; generally this would result in algebraically cancelling pairs of intersections which are fine for the applications below.)

\begin{lemma}
The spheres $\{T_{\ell}\}_{\ell=1}^m$ are mutually disjoint, framed, and embedded in the complement of $A \cup B \cup V$. For some choice of orientations on the $T_{\ell}$ we have algebraic intersection numbers $T_{\ell} \cdot W^{11}_k = \delta_{k\ell}$.  That is, the collection $\{T_\ell\}$ forms a collection of algebraically dual spheres to the Whitney discs $\{W^{11}_k\}$.
    \end{lemma}

\begin{proof}
    The $\delta_{k\ell}$ terms can be seen in \cref{figure:Wh-discs-with-arc-boundaries}.  Each intersection of $\mathring{V}^{11}_{\ell}$ with $\mathring{W}^{11}_k$ gives rise to four intersections between $S_{{V}^{11}_\ell}$ and $\mathring{W}^{11}_k$. These come in algebraically cancelling pairs, so do not contribute to the intersection numbers in the statement of the lemma.
\end{proof}

We write
\[J_{kp} := - \mathring{V}^{11}_{k} \cdot \mathring{W}^{11}_{p} \in \Z\]
and let \[\Sigma_{kp} := \#^{J_{kp}} T_{p}\] denote $J_{kp}$ mutually disjoint parallel, oriented copies of $T_{p}$, tubed together by annuli  contained in a regular neighbourhood of $T_{p}$ that  are disjoint from $A$, $B$, $V$, and $W$. The spheres $\Sigma_{kp}$ are mutually disjoint, framed, and embedded in the complement of $A \cup B \cup V$.
Note that \[\mathring{W}_{\ell}^{11} \cdot \Sigma_{kp} = \mathring{W}_{\ell}^{11} \cdot \#^{J_{kp}} T_{p} = J_{kp} \delta_{p\ell} =  - (\mathring{V}^{11}_{k} \cdot \mathring{W}^{11}_{p})\delta_{p\ell}.\]
Now we want to create a new family of discs whose interiors have trivial algebraic intersection numbers with the interior of every disc $W^{11}_{\ell}$.
We define
\[\wt{V}^{11}_k := V^{11}_k \#_{p=1}^m \Sigma_{kp}.\]
Note that after a small isotopy, the interiors of discs $\{ \wt{V}^{11}_k\}$ may be assumed to be disjoint from the interiors of  $\{ {V}^{11}_k\}$.
We compute
\begin{align} \label{eq: intersection number}
\begin{split}
    \wt{V}^{11}_k \cdot \mathring{W}^{11}_{\ell} &= \mathring{V}^{11}_k \cdot \mathring{W}^{11}_{\ell} + \sum_{p=1}^m \Sigma_{kp} \cdot \mathring{W}^{11}_{\ell} \\  &=  \mathring{V}^{11}_k \cdot \mathring{W}^{11}_{\ell} - \sum_{p=1}^m (\mathring{V}^{11}_{k} \cdot \mathring{W}^{11}_{p})\delta_{p\ell} \\  &= \mathring{V}^{11}_k \cdot \mathring{W}^{11}_{\ell} -  \mathring{V}^{11}_{k} \cdot \mathring{W}^{11}_{\ell} =0.
    \end{split}
\end{align}
Finally, for each $k$ we define
\[\wh{V}^{ij}_k := \begin{cases}
   \wt{V}^{11}_k & (i,j) = (1,1) \\
   V^{ij}_k & (i,j) \neq (1,1).
\end{cases}\]
That is, we replace the discs associated with the innermost eye, but leave all other discs in $\{V^{ij}_k\}$ unchanged.
As we explain below, the new family of discs  $\{ \wh{V}^{ij}_k\}$ will be used
to factorize the given pseudo-isotopy into two, which we can then trivialize separately.

\begin{lemma}\label{lemma:disjointness-wh-V-discs}
The discs $\wh{V}^{ij}_k$ are pairwise disjoint, framed, and embedded.
\end{lemma}

\begin{proof}
Each of the $\wt{V}^{11}_k$ discs lives in a regular neighbourhood of the discs $V^{11}$, together with arcs on $A_1$, say, given by $\partial W^{11} \cap A_1$.  Since the interiors of the boundary arcs of $V^{11}$ and of $W^{11}$ are disjoint by Quinn's arc condition, and since the $V^{ij}$ discs can be assumed to miss a regular neighbourhood of $V^{11}$, we do not create extra intersections.  Also note that the spheres~$\Sigma_{kp}$ are framed, mutually disjoint, disjoint from all $\{ {V}^{ij}_k \}$, and embedded. Hence tubing into them again yields framed and embedded finger discs.
\end{proof}

Next we will use the freedom to introduce additional $S^2\times S^2$ summands
by stabilising  the pseudo-isotopy with $(\#^m S^2 \times S^2) \times [0,1]$. In particular such an operation introduces $m$ additional $ S^2 \times S^2$ summands into the middle-middle level.
This paragraph is specific to the stable proof.
The crucial feature of $\wh V^{11}$ is that the algebraic intersection numbers between the components of $\wh V^{11}$ and the components of $W^{11}$ are all zero. We are free to modify $\wh V^{11}$ as long as they continue satisfying this condition and \cref{lemma:disjointness-wh-V-discs} continues to be satisfied. We take $\wh V^{11}$ as defined above,  
and tube each component into $S^2\times \{\pt\}$ in its own newly added $S^2\times S^2$ summand. Since $\wh V^{11}$ consists of $m$ discs, we add $m$ $S^2 \times S^2$ summands.  Now, in the stabilised middle-middle level,  the entire collection of discs $\wh V^{11}$ has a geometrically dual collection of framed embedded spheres that are disjoint from $A \cup B \cup W \cup (V \sm V^{11})$.

\begin{remark}\label{remark:separate-proof}
    We have structured the proof so that the steps that use stabilisation by copies of $(S^2 \times S^2) \times [0,1]$ are separated from the rest of the proof. This is a device to avoid having to repeat text verbatim during the topological proof in \cref{section:fixing-the-proof-topologically}.
\end{remark}

Now we perform \emph{factorisation}~\cite{Gabai-22}*{Lemma~3.15}.  That is, we introduce \emph{ex nihilo} a cancelling finger-Whitney pair. To explain this we introduce some notation. If we have a finger disc $V_k$, and we use it as a Whitney disc, then we write it as $\ul{V_k}$, to indicate the same disc in the middle-middle level, but with its \emph{modus operandi} altered.  Analogously if we have a Whitney disc $W_{\ell}$ that we wish to use as a finger disc, we write it as $\ul{W_{\ell}}$.

Another way to think about this is to consider two families of discs in the middle level at time~$1/2$. Whitney moves describe the motion of $B$ forward in time, and finger discs describe the motion of $B$ backwards in time.  Underlining a disc indicates reversing the time direction associated with that disc, i.e.\ reversing whether it is considered as a finger or Whitney disc.

As indicated above, the plan is to introduce a trivial family in the middle of the family.  We have a family where finger moves corresponding to the discs $V$ occur at $t_f< 1/2$, and then Whitney moves occur using the discs $W$ at time $t_w > 1/2$.   By a deformation, we may and shall assume that the spheres $A^t$ and $B^t$ are constant during the time interval $[t_f,t_f + 4\varepsilon]$, where $\varepsilon \ll 1/2 - t_f$.

Consider a new family, which proceeds as before until $t_f$.
Then shortly after $t_f$, at time $t_f + \varepsilon$ say, Whitney moves are performed to remove all $A$, $B$  intersections using Whitney discs $\ul{\wh{V}}$.  Then at time $t_f + 3\varepsilon$, these Whitney moves are reversed by finger moves corresponding to the discs $\wh{V}$.  Then for time $t> 1/2$ the family is again as before: at the time $t_w$ the Whitney moves  guided by the discs $W$ are performed, and from time $t_d$ onward the critical points are cancelled.   This new family is related to the original family by a deformation, in which the Whitney moves, and their subsequent undoing, are performed progressively less and less.

We have now accomplished factorisation.  We pass from the finger-Whitney configuration $V\cdot W$ to the configuration $V \cdot \ul{\wh{V}} \cdot \wh{V} \cdot W$.

Note that at times $t_f + \varepsilon<t<t_f + 3 \varepsilon$ the spheres $A_i, B_j$ intersect geometrically in $\delta_{ij}$ points.
Now, deform the family so that all the index 2 and 3 critical points are cancelled at time $t_f + 2 \varepsilon$.

The outcome is two nested eye families.
First we consider the left hand family of nested eyes.  If $(i,j) \neq (1,1)$, then all finger moves involving $A_i$ and $B_j$ are undone with exactly the same disc, acting as a Whitney disc.  Thus we can deform the family to eradicate all of these extra intersections from ever occurring.

\begin{lemma}\label{lemma:innermost-outermost}
    In a family of nested eyes, if there are no handle slides, i.e.\ no $2/2$ nor $3/3$ trajectories, then we can perform a deformation via bigon moves, beak moves, and their inverses, to move the innermost eye  outermost.
\end{lemma}

\begin{proof}
    This is a standard fact in Cerf theory.
It follows from the sequence of beak moves, bigon moves, and their inverses shown in \cref{figure:Cerf-moves-switching-nesting}, in the case of two eyes.  We illustrated the exchange move in Lemma \ref{lemma:innermost-outermost} in the case of two eyes, but it can also be applied to move the innermost eye past an arbitrary number of eyes, to make it outermost. There are no obstructions to performing these moves because by hypothesis there are no trajectories between critical points in different eyes of the same index.
\begin{figure}[ht]
  \centering
\includegraphics[height=8.5cm]{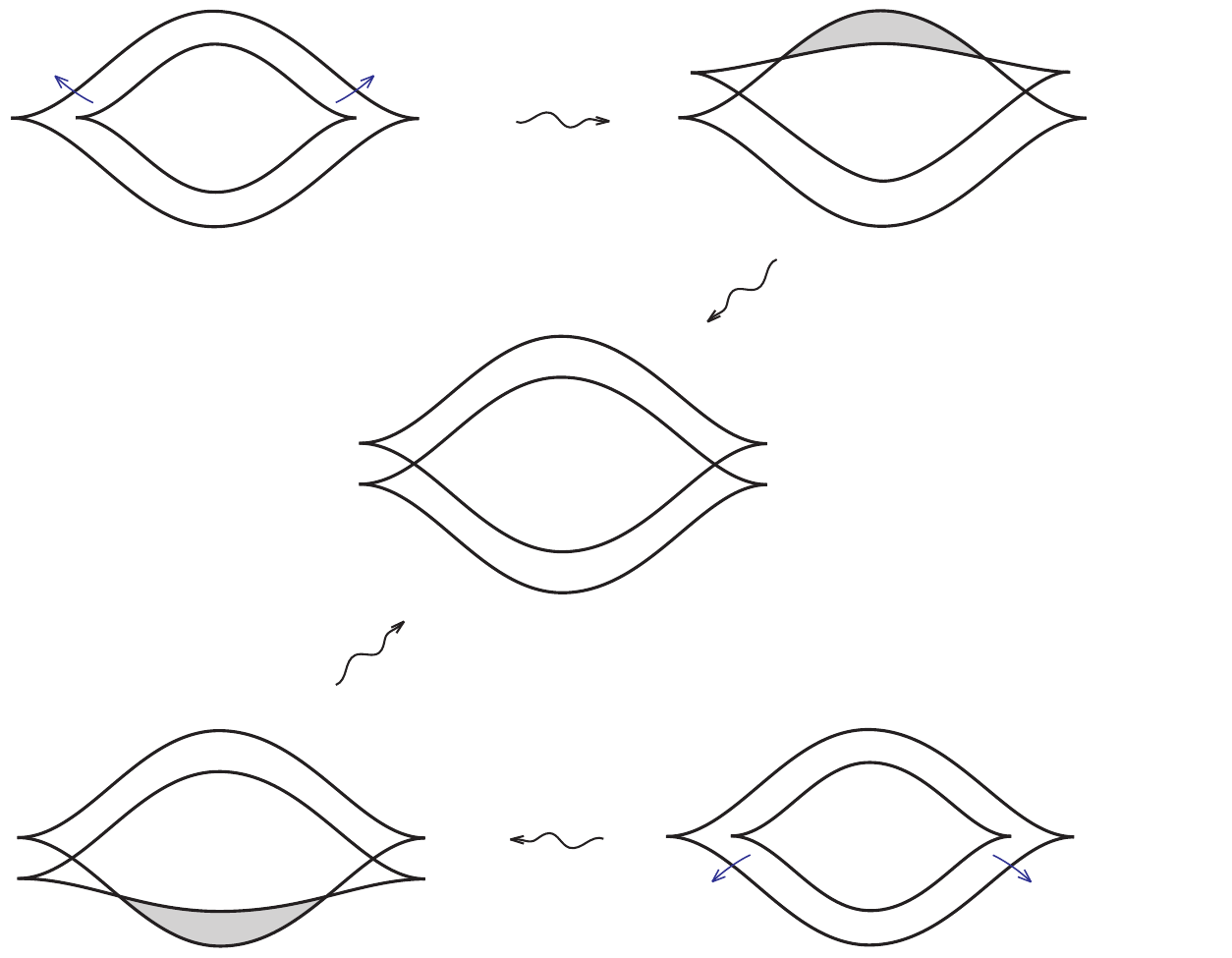}
{\scriptsize
\put(-187, 39){beak move}
\put(-187, 220){beak move}
\put(-118, 160){bigon move}
\put(-265, 80){bigon move}
}
\caption{A sequence of Cerf moves that switch the order of nesting of two concentric eyes. }
\label{figure:Cerf-moves-switching-nesting}
\end{figure}
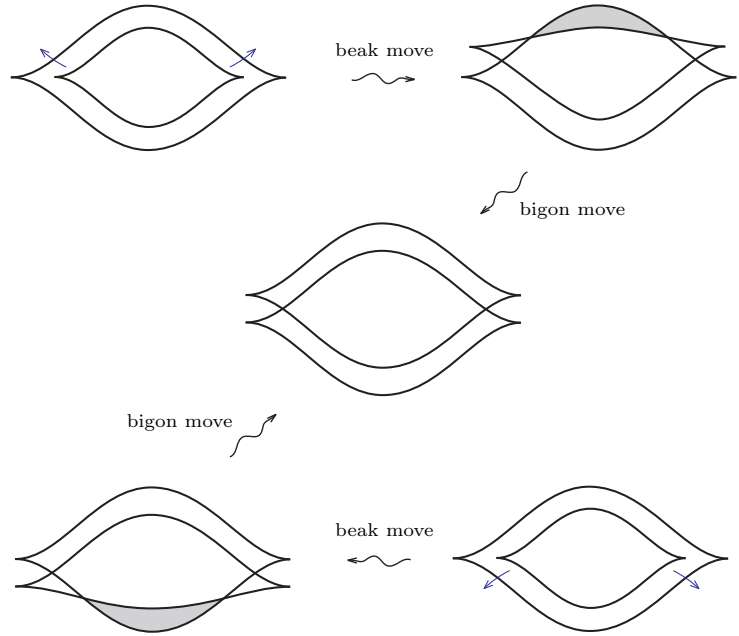
\end{proof}

Applying \cref{lemma:innermost-outermost}, we now remove all eyes apart from the previous innermost eye.
We move the innermost eye outermost, and then all remaining eyes have a unique intersection for all time, i.e.\ $|A^t_i \pitchfork B^t_i| = 1$ for all $t \in [t_b,t_d]$. Hence they can be removed by the case $m=k=1$ of \cite{HW}*{Chapter~V, Proposition~1.1}.

In the right hand nested eye family, by equation~\eqref{eq: intersection number} the finger and Whitney discs for the innermost eye satisfy that the interiors of $\wt{V}^{11}_k$ and $W^{11}_{\ell}$ intersect algebraically trivially, for all $k$, $\ell$.  Therefore the innermost eye of the right hand nested eyes can be removed by \cref{thm:quinn-alg-ints-implies-close-eye-smooth-stable-case}.
We are left with the configuration shown in \cref{figure:From-VW-to-VW’W’W}.
The right hand nested eye family now has $(n-1)$ nested eyes, and so by \cref{ind-hyp}, all of these $(n-1)$ eyes can be removed.

\begin{figure}[ht]
\centering
\includegraphics[height=2.75cm]{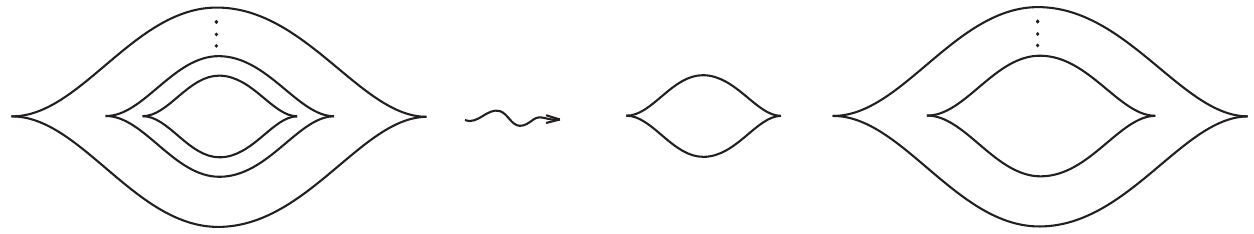}
{\scriptsize
\put(-189,30){$A_1$}
\put(-189,44){$B_1$}
\put(-353,30){$A_1$}
\put(-353,44){$B_1$}
\put(-380,26){$A_2$}
\put(-380,50){$B_2$}
\put(-380,5){$A_n$}
\put(-380,71){$B_n$}
\put(-105,26){$A_2$}
\put(-105,50){$B_2$}
\put(-105,5){$A_n$}
\put(-105,71){$B_n$}
}
\caption{A deformation of the family leads to a modification of the Cerf graphic as shown. The finger/Whitney discs change via factorisation from $V\cdot W$ to $V\cdot \ul{\wh V}\cdot \wh V\cdot W$.
}
\label{figure:From-VW-to-VW’W’W}
\end{figure}
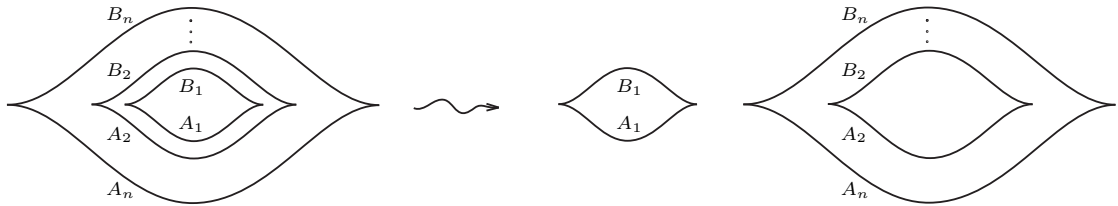

\begin{remark}
The remainder of the proof is specific to the smooth stable case.
\end{remark}

It remains to consider the innermost eye in the left hand family. This is characterised by the collection of (finger, Whitney) discs $(V^{11}, \underline{\wh V}^{11})$.
Here, since $V^{11}$ and $\underline{\wh V}^{11}$ are disjoint from $\bigcup_{i=2}^n (A_i \cup B_i)$, they determine discs in the left hand family's middle-middle level $M \# S^2 \times S^2$, which by abuse of notation we continue to denote with the same symbols.
The unions of the boundaries of $V^{11}$ and $\underline{\wh V}^{11}$  form circles rather than arcs (after a small push of the boundary of $\wh{V}^{11}$ to make the boundaries disjoint).

We would like to convert $(V^{11}, \underline{\wh V}^{11})$ into
finger, Whitney discs satisfying Quinn's arc condition. We will achieve this using the sum square move (\cref{subsection:sum-square}).  First introduce a trivial arc-pair $V_0$, $W_0$ and rename $(V_i, W_i):=(V_{i}^{11}, \underline{\wh V_i}^{11})$, for $i \geq 1$.
Abusing notation, again denote $V:=\cup_i V_i, W:=\cup_i W_i$. We will use sum squares $S_i$ with boundaries on $V_0$ and $V_i$, for each $i \geq 1$, as indicated in Figure \ref{figure:sum square boundaries}.

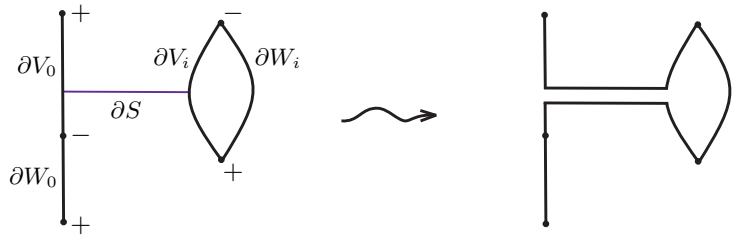
\begin{figure}[h]
\centering
\includegraphics[height=3cm]{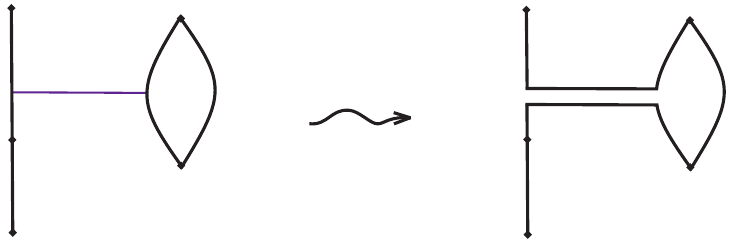}
\put(-250,80){$+$}
\put(-250, 34){$-$}
\put(-250,0){$+$}
\put(-193,80){$-$}
\put(-193, 20){$+$}
{\small
\put(-270,60){$\partial V_0$}
\put(-235,43){$\partial S$}
\put(-273,18){$\partial W_0$}
\put(-219,63){$\partial V_i$}
\put(-181,63){$\partial W_i$}
}
\caption{Rearranging the boundaries of the finger and Whitney discs, on $A$ and on $B$, using the sum square move.}
\label{figure:sum square boundaries}
\end{figure}

Following Quinn~\cite{Quinn:isotopy}*{Sections~4.2~and~4.3}, make the sum square $S_i$  framed and disjoint from~$A$, $B$, and $V$, using spheres dual to $A$ and $B$, that are disjoint from $V$ but intersect $W$.
We add more copies of $S^2\times S^2$, one for each $i$, and pipe each $S_i$ into its own $S^2\times \{\pt\}$.

Now the $S_i$ and the  $W_j = \ul{\wh{V}_j}^{11}$ have a collection of framed embedded geometrically dual spheres, disjoint from everything else and each other,  arising from the new $S^2 \times S^2$ summands.
(See the paragraph following the proof of Lemma \ref{lemma:disjointness-wh-V-discs} for the discussion of dual spheres to $W_j$.)
 Resolve all $S_i$, $W_j$ intersections by tubing $S_i$ into parallel copies of the dual spheres to the $W_j$.  Resolve any self-intersections of $S_i$ by tubing $S_i$ into parallel copies of the dual spheres to $S_i$.  This produces  embedded squares $S_i$ with interiors disjoint from $A \cup B \cup V \cup W$, and correctly framed. Performing the sum square moves creates a family satisfying Quinn's arc condition.  Since $V$ and $\wh{V}$ had disjoint interiors, by the sentence preceding equation~\eqref{eq: intersection number}, the discs in the new family after the sum square move also all have pairwise disjoint interiors.

This allows us to stably, smoothly trivialize the pseudo-isotopy given by $(V, \underline{\wh V})$.
Thus the remaining eye, the formerly innermost eye in the left hand family, can also be removed.
This completes the proof of the inductive step, and hence completes the proof of \cref{thm:PI-stable-I}.
\qed

\section{A decomposition theorem for topological 4-manifolds}
\label{sec:Perron}

We prove a decomposition lemma (\cref{lemma:perron-7-2-plus-details}) for compact 1-connected topological 4-manifolds.  The proof is an adaptation of the proofs in \cite{Perron}*{Section~7} which Perron attributes to Siebenmann.  The argument given there is for closed 1-connected 4-manifolds, and  relies on the Freedman-Quinn classification of 1-connected  closed topological 4-manifolds.
 We explain how this argument can be adapted to compact 1-connected 4-manifolds, with nonempty boundary, using Boyer's classification~\cites{Boyer86,Boyer93} of compact 1-connected topological 4-manifolds.
This has two aims.
\begin{enumerate}[(i)]
    \item\label{aim-i} \cref{thm:main-PI-I} stated by Quinn is more general than the result of Perron, because it encompasses all compact 1-connected topological 4-manifolds with nonempty boundary. Perron's result in the case of nonempty boundary has the assumption that the 4-manifold is smooth and has a handle decomposition with no 1-handles.  We clarify that  \cref{thm:main-PI-I} can be deduced from Perron's result, using the decomposition lemma proven in this section.
    \item\label{aim-ii} The first step of Quinn's proof of \cref{thm:main-PI-I} uses a decomposition result, \cite{Quinn:isotopy}*{Proposition~1.5}. The proof given in this section, following Siebenmann's argument, seems to us to be easier than Quinn's proof in \cite{Quinn:isotopy}*{Section~6} of his Proposition 1.5, so this represents a simplification of Quinn's proof.
\end{enumerate}

\begin{remark}
    Both Perron and Quinn remark that \cref{thm:main-PI-I} might follow from \cite{Perron}*{Theorem~0} and \cite{Quinn:isotopy}*{Proposition~1.5}, via the argument of \cite{Perron}*{Section~7}.  However, \cite{Quinn:isotopy}*{Proposition~1.5} allows 1-handles, and so we do not believe this is the case.  Boyer's classification \cites{Boyer86,Boyer93}, which was completed after \cite{Quinn:isotopy}, avoids 1-handles, as stated in \cref{thm:boyer} below.
\end{remark}

We will make use of the following result of Boyer~\cites{Boyer86, Boyer93}; see in particular \cite{Boyer93}*{p.~35}.

\begin{theorem}[Boyer]\label{thm:boyer}
    Let $M$ be a 1-connected 4-manifold with connected nonempty boundary~$Y$. Then $M$ is homeomorphic to a manifold of the form $Y \times [0,1] \cup H \cup C$, where $H$ consists of finitely many 2-handles attached to $Y \times \{1\}$, such that  $\partial(Y \times [0,1] \cup H) = Y \times \{0\} \sqcup \Sigma$, where $\Sigma$ is a $\Z$-homology 3-sphere, and  $C$ is a compact, contractible 4-manifold with $\partial C \cong \Sigma$.
\end{theorem}

Since Boyer's result requires that $\partial M = Y$ is connected, first we dispense with the case that $\partial M$ is disconnected.

\begin{lemma}\label{lemma:disconnected-bdy}
    Suppose that every pseudo-isotopy is isotopic rel.\ $\sqsubset$ to the identity for  1-connected compact 4-manifolds with connected boundary. Then every pseudo-isotopy is isotopic rel.\ $\sqsubset$ to the identity for all 1-connected compact 4-manifolds.
\end{lemma}

\begin{proof}
    Let $M$ be a 1-connected compact 4-manifold with boundary components $Y_1,\dots,Y_n$. Let $F \colon M \times [0,1] \to M \times [0,1]$ be a pseudo-isotopy and let $f:= F|_{M \times \{1\}}$. Choose pairwise disjoint locally flat embedded arcs $\gamma_2,\dots,\gamma_{n}$ in $M$, such that for each $\gamma_i$, one endpoint is on $Y_1$ and the other endpoint is on $Y_i$. The image $F(\gamma_i \times [0,1])$ is a concordance from $\gamma_i$ to $f(\gamma_i)$. Observe that (forgetting the $[0,1]$ coordinate in $M \times [0,1]$) concordance implies homotopy, and homotopy implies isotopy for arcs in a 4-manifold, so $F$ is isotopic to a homeomorphism, again denoted $F$, such that $F(\gamma_i \times [0,1])$ is the trace of an isotopy, for each $i$.  By a further isotopy we can assume $F$ fixes each $\gamma_i \times [0,1]$ pointwise, and in fact fixes an open neighbourhood of each $\gamma_i \times [0,1]$ pointwise, $i=2,\dots,n$.  Removing these neighbourhoods and restricting $F$ to their complement yields a pseudo-isotopy $F|$ of a 1-connected 4-manifold with connected boundary.  By hypothesis $F|$ is isotopic rel.\ $\sqsubset$ to the identity.  Extending this isotopy over the removed neighbourhoods with the identity map yields an isotopy of the original $F$ to the identity of $M \times [0,1]$, rel.\ $\sqsubset$.
\end{proof}

From now on we assume that $\partial M =Y$ is connected.  The following Decomposition Lemma and its proof are modelled on \cite{Perron}*{Lemma~7.2}. We check that by invoking Boyer's classification  the proof extends to the case of nonempty boundary, providing some necessary details.
The statement is closely related to \cite{Quinn:isotopy}*{Proposition~1.5}.

\begin{lemma}[Decomposition Lemma] \label{lemma:perron-7-2-plus-details}
    Assume that $M^4 = Y \times [0,1] \cup H \cup C$ is as in \cref{thm:boyer} and that $\partial M = Y$ is connected.
    The contractible manifold $C$ is contained in the interior of a topologically embedded 4-ball $D^4 \subseteq M$.  It follows that we can write $M$ as $M = Y \times [0,1] \cup H \cup D^4$, where $H$ consists of finitely many 2-handles attached to $Y \times \{1\}$.
\end{lemma}

\begin{proof}
Let $C_1$ and $C_1'$ be two copies of $C$. The union $C_1 \cup_{\partial} (-C_1')$ is a homotopy sphere and so is homeomorphic to $S^4$~\cite{F}.
It follows that $C_1'\subseteq S^4$, and (removing a $4$-ball from $S^4$) $C_1'\subseteq D^4$.
Considering a $4$-ball $D^4$ in the interior of $C$, we have inclusions
\begin{equation}\label{eqn:inclusions}
    C_1' \subseteq D^4 \subseteq C \subseteq M.
\end{equation}
    Let $W:= M \sm \mathring{C}$ and let $W' := M \sm \mathring{C}_1'$. We claim there exists a homeomorphism rel.\ boundary $f \colon W \xrightarrow{\cong} W'$. Assuming the claim, since every homeomorphism of a homology 3-sphere $\Sigma$ extends over any given contractible 4-manifold\footnote{This follows from Boyer's classification~\cite{Boyer86}, but can also be proved more directly. Let $f \colon \Sigma \xrightarrow{\cong} \Sigma$, and fix $C^4 \simeq \{\ast\}$ with $\partial C=\Sigma$. Then $D:= C \cup_f -C$ is a homotopy 4-sphere, so by~\cite{F} $D \cong S^4 = \partial D^5$. Consider $D^5$ as a rel.\ boundary $h$-cobordism from $C$ to $C$ rel.\ $f$, and apply the $h$-cobordism theorem~\cite{FQ}*{Theorem~7.1A}.} with boundary $\Sigma$, this extends to a homeomorphism $\ol{f} \colon M = W' \cup C_1' \xrightarrow{\cong} W \cup C = M$ that sends $C_1'$ to $C$. For the copy of $D^4$ as in \eqref{eqn:inclusions}, the image $\ol{f}(D^4)$ is a 4-ball in $M$ whose interior contains $C$.

It remains to prove the claim: for a fixed 1-connected, compact $M$, removing two different copies of the same contractible 4-manifold, $C$ and $C_1'$, from $M$, yields homeomorphic 4-manifolds.
To prove it, we apply Boyer's classification from~\cite{Boyer86}. First note we can apply the tubing trick from the proof of \cref{lemma:disconnected-bdy}, together with homotopy implies isotopy for arcs, to reduce to the case of connected boundary. We assert that Boyer's classifying invariants coincide for a pair of 4-manifolds if and only if they coincide after gluing a contractible 4-manifold to both along homeomorphic homology 3-discs in the boundaries. Boyer's invariants coincide after gluing a contractible 4-manifold because both manifolds $W \cup C$ and $W' \cup C_1'$ are copies of~$M$.

It remains to see the assertion that Boyer's classifying invariants are not affected by gluing contractible manifolds along homology 3-discs. Boyer's main invariant is the intersection form of the 4-manifold~$M$, together with an isomorphism of its algebraic boundary with the linking form on the torsion submodule of $H_1(Y)$. His secondary invariant only applies for a pair of spin 4-manifolds, and (after fixing a spin structure on $Y$) lies in a subgroup of $H^1(Y;\Z/2)$. It records the spin  structure induced on $Y$ by the 4-manifold.  These homological invariants are not affected by gluing on a contractible manifold, and so the assertion holds.
\end{proof}

As promised in \eqref{aim-i} at the start of the section, we deduce \cref{thm:main-PI-I} using Perron's proof from \cite{Perron}*{Sections~3--5}, which culminates in \cite{Perron}*{Lemme~5.2}, together with the Decomposition \cref{lemma:perron-7-2-plus-details}.

\begin{proof}[Proof of \cref{thm:main-PI-I} using \cite{Perron}]
       By \cref{lemma:disconnected-bdy} we assume  without loss of generality that $\partial M$ is connected.
By the Decomposition \cref{lemma:perron-7-2-plus-details}, we can write $M$ as $M = Y \times I \cup H \cup D^4$, where $I = [0,1]$ denotes the collar coordinate.

Let $F \colon M \times [0,1] \to M \times [0,1]$ be a pseudo-isotopy. Recall that $F|_{\sqsubset} = \Id_{\sqsubset}$. Isotope $F$, using uniqueness of collars, to arrange that $F$ restricted to a collar on $\sqsubset$ is the identity, and in particular is the identity on $(M \times \{0\}) \cup (Y \times I \times [0,1])$.

Now we consider the cores $V := \sqcup_i (D^2 \times \{0\})_i \subseteq \sqcup_i (D^2 \times D^2)_i = H$ of the 2-handles $H$.  If we were to turn the 2-handles upside down, then these would be the cocores of the upside-down 2-handles.

The main step of Perron's proof, which culminates in the statement of \cite{Perron}*{Lemme~5.2} is to take a tubular neighbourhood $\mathcal{N}$ of the union $V$ of the cocores of his 2-handles (the cores of our 2-handles $H$), and perform an isotopy of the pseudo-isotopy $F$,  rel.\ the collar on $\sqsubset$, to arrange that its restriction to this neighbourhood  is the inclusion $\mathcal{N} \times [0,1] \to M \times [0,1]$. Possibly after a further isotopy, we may take~$\mathcal{N}=H$, and hence as a result we have a pseudo-isotopy  $F' \colon M \times [0,1] \to M \times [0,1]$ that is the identity away from $\mathring{D}^4 \times [0,1] \subseteq M \times [0,1]$, and is the identity on $D^4 \times \{0\}$. By Alexander's coning trick, with cone point $(0,1) \subseteq \mathring{D}^4 \times [0,1]$, we can further isotope $F'$ rel.\ $(M \times [0,1]) \sm (\mathring{D}^4 \times [0,1])$ to the identity.
\end{proof}

We go on to explain our fix of Quinn's proof of \cref{thm:main-PI-I}, which, as stated in \eqref{aim-ii}, starts with an application of \cref{lemma:perron-7-2-plus-details}.

\section{The topological pseudo-isotopy theorem}\label{section:fixing-the-proof-topologically}

We start by recalling Quinn's strategy for the proof of \cref{thm:main-PI-I}, which begins in \cite{Quinn:isotopy}*{Section~5}. Our fix starts after the statement of \cref{thm:quinn-alg-ints-implies-close-eye}.

Let $M$ be a compact, 1-connected topological 4-manifold and let $F \colon M \times [0,1] \to  M \times [0,1]$ be a topological pseudo-isotopy, i.e.\ a homeomorphism that is the identity on $\sqsubset$. By \cref{lemma:disconnected-bdy} we may and shall assume that $\partial M =: Y$ is connected.
By the Decomposition \cref{lemma:perron-7-2-plus-details}, $M = Y \times [0,1] \cup H \cup D^4$, where $H$ consists of 2-handles added to $Y \times \{1\}$. Define \[X:= Y \times [0,1] \cup H \subseteq M,\] and note that $\partial M = Y \times \{0\} \subseteq X$.  The goal, as on \cite{Quinn:isotopy}*{p.~363}, is to prove the following.

\begin{proposition}\label{prop:goal}
    The pseudo-isotopy $F$ is isotopic rel.\ $\partial M = Y \times \{0\}$ to a pseudo-isotopy $F'$ that is the identity on $X$.
\end{proposition}

Since $M \sm X \subseteq D^4 \subseteq M$, we have that $(M \sm X) \times [0,1] \subseteq D^4 \times [0,1] \subseteq M \times [0,1]$, we can apply the Alexander trick to $F'$, pushing towards $(0,1) \in D^4 \times [0,1] \subseteq M \times [0,1]$ as the cone point, to isotope $F'$ to the identity. Hence it suffices to prove \cref{prop:goal}.

Next, Quinn assumes, after an isotopy, that $F$ is the identity on $B \times [0,1]$, where $B \subseteq M \sm X$ is a 4-ball.  Write $M_0$ for $M$ with the origin of $B$ removed.

Our next goal is to smooth the restriction of our pseudo-isotopy $F| \colon M_0 \times [0,1] \to M_0 \times [0,1]$, with respect to some smooth structures on the domain and codomain. These smooth structures will not coincide, in general.  We use results of Perron and Quinn, while providing additional details on how their results apply to smooth $F|$.

By \cite{FQ}*{Theorem~8.2}, there is a smooth structure $\sigma$ on $M_0$.  Take the product $\Sigma$ of this with the standard smooth structure on $[0,1]$, to obtain the smooth 5-manifold $\{M_0 \times [0,1]\}_{\Sigma}$, where the subscript denotes the choice of smooth structure.
Note that the projection $\pr_2 \colon \{M_0 \times [0,1]\}_{\Sigma} \to [0,1]$ is a smooth Morse function with no critical points, and that $F$ restricts to a homeomorphism $F| \colon M_0 \times [0,1] \to M_0 \times [0,1]$ that is the identity on $((B \sm \{0\}) \times [0,1]) \cup \sqsubset$.

Recall~\cite{Kirby-Siebenmann:1977-1}*{Essay~II,~\S 0,~p.~57} that a smooth structure $\theta$ on $M_0 \times [0,1]$ is \emph{sliced} if the projection $\pr_2 \colon \{M_0 \times [0,1]\}_{\theta} \to [0,1]$ is a smooth submersion.
Perron showed the following, making use of immersion theory~\cites{Lees,Lashof-70-I,Lashof-70-II,Lashof-71}, and Quinn's prior result~\cite{Quinn-annulus}*{Corollary~2.2.3} that $\TOP(4)/\OO(4) \to \TOP/\OO$ is 3-connected.

\begin{theorem}[{\cite{Perron}*{Corollaire~1.3}}]
There exists a sliced smooth structure $\theta$ on $M_0 \times [0,1]$ such that
 $F|_{M_0 \times [0,1]}$ is isotopic  to a smooth pseudo-isotopy
 \[ F' \colon \{M_0 \times [0,1]\}_{\Sigma} \to \{M_0 \times [0,1]\}_{\theta}\]
 rel.\ $(\partial M \times  [0,1]) \cup (M_0 \times \{0\})$.
\end{theorem}

Note that $\pr_2 \circ F' \colon \{M_0 \times [0,1]\}_{\Sigma} \to [0,1]$ is a smooth Morse function with no critical points. Here we use that $F'$ is a diffeomorphism and that $\pr_2$ is a smooth submersion with respect to $\theta$. Now we can apply Cerf theory, as at the start of \cref{sec:setting-up}, to obtain a corresponding 1-parameter family of generalised  Morse functions and gradient-like vector fields $(G_t,\xi_t)$  on $M_0 \times [0,1]$.
Since we are in the non-compact setting there can be infinitely many critical points.
However note that only finitely many of them intersect $X$ for some $t$, because $X$ and $[0,1]$ are compact.
Ultimately, it will suffice to remove this finite subset of the critical points.

Next, we describe how Quinn deals in his argument with the fact that $(G_t,\xi_t)$ can potentially have infinitely many critical points.
First, Quinn defines a control function $\gamma \colon M_0 \to [0,\infty)$ to be $0$ on $M \sm B$ and $(1/d(x,0))-1$ on $B$, where $d$ denotes the standard Euclidean metric on the ball $B$.  Since $F'$ arises from a pseudo-isotopy of the compact $4$-manifold $M$, the family $(G_t,\xi_t)$ can be assumed to be controlled, i.e.\ to have handles of controlled diameter, meaning each of their images under $\gamma$ has diameter less than some given constant.
There is then a controlled Reduction to Eyes Lemma~\cite{Quinn:isotopy}*{Lemma~5.1}, proven in \cite{Quinn:isotopy}*{Section~5.3}.
The stated outcome is a one parameter family of handle structures where each of the handles has controlled diameter, that consists only of 2- and 3-handles, with independent births and deaths and no handles slides. Implicit in the statement, as indicated by the name of the lemma and by the way the lemma is applied in \cite{Quinn:isotopy}*{Section~5}, is that additionally the associated Cerf graphic  consists of nested eyes.

Here is a brief outline of how Quinn proved this, building on the work of Hatcher--Wagoner and his own work on high-dimensional controlled pseudo-isotopy~\cite{Ends-of-maps-IV}.
For the convenience of the reader, we provide a more specific list of statements in \cite{HW} than was given in \cite{Quinn:isotopy}.  Quinn's argument applies the following methods from Hatcher--Wagoner, but in the non-compact controlled setting.
It is shown in \cite{HW}*{Chapter~V,~\S 3} how to restrict to families with critical points of indices only~2 and~3, i.e.\ 1-parameter families of handle structures with only 2- and 3-handles.  Elimination of swallowtails~\cite{HW}*{Chapter~V,~\S 5} and the Independent Trajectories Principle~\cite{HW}*{Chapter~I,~\S 7, pp.~64--7} is applied, as in the proofs of \cite{HW}*{Chapter~VI, Propositions~3 and 4, pp.~214--7}, to obtain eyes consisting only of 2- and 3-handles, with independent births and deaths, but possibly with 2/2 and 3/3 rearrangements and handle slides.   Hatcher--Wagoner's Exchange Lemma \cite{HW}*{Chapter~IV, Lemma~2.1, p.~132} trades all the 2/2 events for 3/3 events; see also \cite{HW}*{pp.~142--3}.  Until this point, the Hatcher--Wagoner arguments in the compact case generalise to the non-compact, controlled setting, as originally observed in~\cite{Ends-of-maps-IV}. This need not hold for the Hatcher--Wagoner argument to remove 3/3 handle slides and rearrangements. Instead, the argument of \cite{Quinn:isotopy}*{Section~5.3} (after the first two paragraphs) accomplishes this, while preserving the condition on controlled diameter of handles.

As mentioned above, there are finitely many critical points, $n$ say, that intersect $X$ for some $t \in [0,1]$.
Since there are no handle slides and the births and deaths are independent, we can apply the argument in \cref{lemma:innermost-outermost} to perform a deformation of the family such that the $n$ eyes intersecting $X$ are the $n$ innermost eyes in the nesting.

Again the strategy is to remove one of these $n$ eyes at a time, starting with the innermost eye.
The proof is by induction on $n$, using the following inductive hypothesis, which is the topological, non-compact version of \cref{ind-hyp}.

\begin{indhyp}[$\mathbf{E(n)}$]
 \label{ind-hyp-top}
     Let $M$, $X \subseteq M$, and $M_0$ be as above.  Let $(G_t,\xi_t)$ be  a 1-parameter family of generalised Morse functions and gradient like vector fields on $M_0 \times [0,1]$. Suppose that the associated Cerf graphic consists of nested eyes 
     with cancelling pairs of index 2 and 3 critical points, with no handle slides, no rearrangements, and independent births and deaths.
   Suppose  there are at most $n$ eyes that, for some time $t$, intersect $X \times [0,1]$, and that these are the $n$ innermost eyes in  the Cerf graphic.
     Then there is a deformation of $(G_t,\xi_t)$ to a 1-parameter family without any critical points that intersect $X \times [0,1]$.
 \end{indhyp}

Quinn's controlled Reduction to Eyes Lemma~\cite{Quinn:isotopy}*{Lemma~5.1} implies that the hypotheses of $E(n)$ hold for some $n$.  We will prove, by induction, that $E(n)$ holds for all $n$. As a result, we will obtain an isotopy of our pseudo-isotopy $F' \colon M_0 \times [0,1] \to M_0 \times [0,1]$ to one that is the identity on $X \times I$, and that is controlled. Due to the latter condition, this extends continuously to $M \times [0,1]$, which completes the proof of \cref{prop:goal} modulo the upcoming induction.

To prove the base case $E(1)$ and the inductive step, we will use the notation from \cref{sec:setting-up}, which we briefly recall.
We have birth time $t_b$, and a death time $t_d$.
In the level sets $G_{t}^{-1}(1/2)$, for $t_b < t< t_d$, we have  ascending spheres $\{A_i^t\}$ of the index 2 critical points, and $n$ descending spheres $\{B_i^t\}$ of the index 3 critical points.
The $A_i^t$ and the $B_i^t$ are enumerated so that the $n$ eyes that intersect $X$ are first, with $A_1^t$ and $B_1^t$ correspond to the innermost eye among these $n$ eyes, and then moving outwards. The spheres $A_1^t,\dots, A_n^t$ and $B_1^t,\dots,B_n^t$ can intersect $X \times [0,1]$, whereas $A_i^t$ and $B_j^t$ are disjoint from $X \times [0,1]$, for all $i,j>n$ and  for all $t \in [0,1]$.

Let $A^t := \cup_{i=1}^n A_i^t$ and $B^t := \cup_{i=1}^n B_i^t$.
We  also have a finger move time $t_f$ and a Whitney move time $t_w$, with $t_b < t_f < 1/2 < t_w < t_d$.
In the middle-middle level $G_{1/2}^{-1}(1/2)$ we see the spheres $A_i := A_i^{1/2}$,  $B_j := B^{1/2}_j$, the finger move discs  $\{V^{ij}_k\}$ and  the Whitney discs $\{W^{ij}_{\ell}\}$. We write $V^{ij} := \cup_k V^{ij}_k$, $W^{ij} := \cup_{\ell} W^{ij}_{\ell}$, $V := \cup_{i,j} V^{ij}$, and $W := \cup_{i,j} W^{ij}$.

The fact that $A_i^t$ and $B_j^t$ are disjoint from $X \times [0,1]$, for all $i,j>n$ and for all $t \in [0,1]$,  implies that in the middle-middle level all the discs $V^{ij}$ and $W^{ij}$ are disjoint from $X$ for $j > n$; an intersection of one of these discs with~$X$ would entail a motion of $B_j^t$, for some $j >n$, through $X \times [0,1]$.

The following statement summarizes the result of
\cite{Quinn:isotopy}*{Sections~4.6~and~5.2}.
This is the topological analogue of \cref{thm:quinn-alg-ints-implies-close-eye-smooth-stable-case}, proved using the disc embedding theorem.

\begin{theorem}[Quinn]\label{thm:quinn-alg-ints-implies-close-eye}
 In the situation described above, if Quinn's arc condition is satisfied for $A_1,B_1$, and if the algebraic intersection numbers $\mathring{V}^{11}_k \cdot \mathring{W}^{11}_{\ell}$ vanish for all $k,\ell$, then there is a topological deformation of the family to one with $(n-1)$ nested eyes that intersect $X \times [0,1]$.
\end{theorem}

We need to arrange a situation that we can apply this result to prove the inductive step.
As in the smooth stable case discussed in \cref{section:fixing-the-proof}, the problem with the proof in \cite{Quinn:isotopy} is in the use of the replacement criterion to arrange $\mathring{V}^{11}_k \cdot \mathring{W}^{11}_{\ell}=0$, to arrange that the hypotheses of \cref{thm:quinn-alg-ints-implies-close-eye} are satisfied. We give an alternative argument below.

As a result of the use of control theory, discussed above, $V^{ij}$ and $W^{ij}$ are disjoint from $X$ for $j > n$.
Quinn's primary concern in \cite{Quinn:isotopy}*{Section~5.2} is to preserve this condition as an eye is cancelled.  This guarantees that the associated eyes, with enumeration $>n$, remain disjoint from $X \times [0,1]$.  During this proof, Quinn mentions the replacement criterion in the second paragraph of \cite{Quinn:isotopy}*{p.~366}. Since we do not use the replacement criterion in our fix, this paragraph becomes irrelevant. We observe that our new proof preserves $V^{ij}\cap X=W^{ij}\cap X=\emptyset$  for $j > n$.

As in \cref{section:fixing-the-proof}, we perform the factorisation to replace the family with discs $V \cdot W$ with the concatenation of moves corresponding to discs $V \cdot \ul{\wh{V}} \cdot \wh V \cdot W$. The discs $\wt{V}^{11}$ are defined as before, by tubing into Whitney spheres. Ignore the modification to these discs from the paragraph just above \cref{remark:separate-proof}; as stated in that remark, this modification was specific to the smooth stable proof.   In particular, we do not stabilise the pseudo-isotopy with $(\#^m S^2\times S^2) \times [0,1]$.
Then for each $k$ we define
\[\wh{V}^{ij}_k := \begin{cases}
   \wt{V}^{11}_k & (i,j) = (1,1) \\
   V^{ij}_k & (i,j) \neq (1,1).
\end{cases}\]
We emphasise that in the second case we make this definition for all the possibly infinitely many values of  $(i,j) \neq (1,1)$.  The factorisation involves all the eyes, not just the $n$ innermost eyes that we are trying to cancel.   The only discs that are modified are the $V^{11}$, and these already intersect~$X$. So we continue to have $\wh{V}^{ij}_k \cap X = V^{ij}_k \cap X = \emptyset$ for $j>n$ and for all $k$.

As in \cref{section:fixing-the-proof}, we cancel all the pairs of index $2-3$ critical points for a small period of time in the middle, and consider again the two families of (possibly infinitely many) nested eyes corresponding to a 1-parameter family of Morse data on $M_0 \times [0,1]$.

By the same argument as in \cref{section:fixing-the-proof}, we remove all but the innermost eye in the left hand family of nested eyes.  We will return to analysing this eye later.

In the innermost eye of the right hand family, the finger and Whitney discs once more satisfy the algebraic intersection hypothesis, as well as Quinn's arc condition. Thus by \cref{thm:quinn-alg-ints-implies-close-eye} we can remove the innermost eye without causing any new critical points to intersect $X \times [0,1]$, and maintaining control.  Then the right hand family contains $(n-1)$ nested eyes that intersect $X \times [0,1]$. If $n=1$ there are no eyes left and so we are done for the right hand family. Otherwise, by \cref{ind-hyp-top} we can remove  $(n-1)$ nested eyes by a deformation, without causing any other critical points to intersect $X \times [0,1]$, in the family that is the outcome of this deformation.

Now we analyse the left hand family, consisting of one remaining eye. This has finger discs $V^{11}$, and then Whitney discs $\underline{\wh{V}}^{11}$ made from tubing the discs in $V^{11}$ into Whitney spheres.  We do not need to worry about control here, nor intersections with $X \times [0,1]$, because we have reduced this family to a single eye. We just need to see that the pseudo-isotopy corresponding to this family is topologically isotopic to the identity.

Let $t_{b+}$ be a time shortly after the birth time $t_b$.  Then $A^{t_{b+}}$ and $B^{t_{b+}}$ denotes the ascending and descending spheres, respectively, of the index 2 and 3 critical points, in the middle level $G_{t_{b+}}^{-1}(1/2)$, shortly after the birth.
Fixing a regular neighbourhood of $A^{t_{b+}} \cup B^{t_{b+}}$ determines a decomposition of $G_{t_{b+}}^{-1}(1/2)$ as $M_0 \# S^2 \times S^2$.

\begin{lemma} \label{lemma:isotopy of finger moves}
    We may assume, after a deformation of the 1-parameter family, that every finger move,  when viewed in the middle-middle level,  occurs entirely within the $\# S^2 \times S^2$ summand of $M_0 \# S^2 \times S^2$.
\end{lemma}

\begin{proof}
  Since $M_0 \# S^2 \times S^2$ is simply-connected and the union of the ascending and descending spheres of the critical points are $\pi_1$-negligible, the complement $N:= \big(M_0 \# S^2 \times S^2\big) \sm \big(A^{t_{b+}} \cup B^{t_{b+}}\big)$ is simply-connected.

  Consider the collection of finger move arcs, with ends joined together via disjoint arcs on $A^{t_{b+}}$ and $B^{t_{b+}}$. By the previous paragraph these loops are null homotopic in the complement of the ascending and descending spheres, and hence bound immersed discs with interiors in $N$.
See the left picture in \cref{figure:Finger_move}. We can push intersections between these discs, and self-intersections, off the boundary, over the part of the boundary of the discs consisting of the finger move arc.  We obtain a collection of mutually disjoint embedded discs.  We can use these to isotope all the finger moves into an arbitrarily small neighbourhood of $A^{t_{b+}} \cup B^{t_{b+}}$. We already chose such a neighbourhood in order to determine the decomposition of $G_{t_{b+}}^{-1}(1/2)$ as $M_0 \# S^2 \times S^2$.  Hence we can shrink the finger move by an isotopy to lie in the $\# S^2 \times S^2$ summand.
\begin{figure}
\centering
\includegraphics[height=3.2cm]{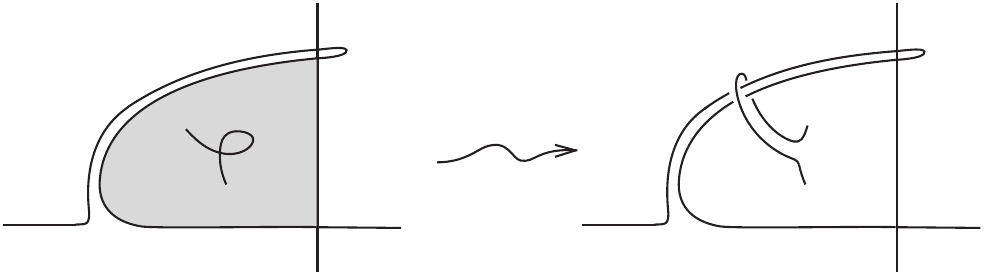}
\caption{The finger move arc bounds an embedded disc.  }
\label{figure:Finger_move}
\end{figure}
This change can be realised by a deformation of the whole family, because we can apply an ambient isotopy consistently for all times before $t_f$.
\end{proof}

By \cref{lemma:isotopy of finger moves}, we assume that in our left hand family, each finger move, with finger disc $V^{11}_k$, occurs within the $S^2 \times S^2$ factor in the middle level.
Also each Whitney move, with Whitney disc $\ul{\wt{V}}^{11}_k$, occurs in a neighbourhood of  $A_1 \cup B_1 \cup \wt{V}^{11}$.   We built the Whitney discs $\ul{\wt{V}}^{11}_k$ using Whitney spheres, and these can be assumed to lie in an arbitrary neighbourhood of $V^{11}$.  Hence each Whitney move occurs in a neighbourhood of  $A_1 \cup B_1 \cup V^{11}$.

We assumed that $A_1$ stays fixed, and we have now arranged that all motions of $B_1$ occur within the $S^2 \times S^2$ summand.  Hence all motions of $A_1$ and $B_1$ stay within this summand.
This translates to a pseudo-isotopy supported in a copy of $D^4 \subseteq M_0$.  That is, we have an isotopic pseudo-isotopy $H \colon M_0\times [0,1] \to M_0 \times [0,1]$, a decomposition
\begin{equation}\label{eqn:decomposition}
    M_0 \times [0,1] \cong (M_0 \sm \mathring{D}^4) \times [0,1] \cup_{\partial D^4 \times [0,1]} D^4 \times [0,1],
\end{equation}
and a pseudo-isotopy $P \colon D^4 \times [0,1] \to D^4 \times [0,1]$, such that $H$ has the form
\[H = \Id \cup P\]
with respect to the decomposition~\eqref{eqn:decomposition}.
By Alexander's coning trick, again with cone point $(0,1) \subseteq D^4 \times [0,1]$, the pseudo-isotopy $P$ is topologically isotopic rel.\ $\sqsubset$ to $\Id_{D^4 \times [0,1]}$.  Hence $H$ is isotopic to $\Id_{M_0 \times [0,1]}$, and the pseudo-isotopy corresponding to our left hand family is topologically isotopic to $\Id_{M_0 \times [0,1]}$.
This completes the proof of the inductive step, and hence completes the proof of \cref{thm:main-PI-I}.
\qed

\section{The disc replacement criterion}\label{section:DRC-problem}

We discuss the problems with the given proof of the replacement criterion in \cite{Quinn:isotopy}*{Section~4.5}, and mention a couple of related open problems.

We begin with a meta argument that there ought to be a problem with given proof of the replacement criterion.
The proof given in \cite{Quinn:isotopy}*{Section~4.5} does not use the simply-connected hypothesis, nor the hypothesis that the boundaries of finger and Whitney discs are arranged in an arc.  The same proof, if valid, would allow one to drastically alter the Hatcher--Wagoner secondary obstruction, in particular in the $\Wh_1(\pi_1(M);\pi_2(M))$ summand, leading to a contradiction with the work of Hatcher--Wagoner and Igusa~\cites{HW,Igusa-what-happens}.

\subsection{Description of the problem}

Now we highlight a specific issue with the proof given by Quinn in \cite{Quinn:isotopy}*{Section~4.5} and also illustrate a connection between Quinn's idea and the method of factoring used in our corrected proof. We give a thorough account of the approach in \cite{Quinn:isotopy}*{Section~4.5} because the exposition in that reference is short on detail, so that pinpointing the error is not immediate. Since the replacement criterion remains an open problem, stated as \cref{problem:DRC}  below, we think it is important to explain the precise problem with the original approach.

The goal of \cite{Quinn:isotopy}*{Section~4.5} is to prove that in a single eye, with $A$, $B$, $V$, and $W$ satisfying Quinn's embedded arc condition, if there is a collection of alternative Whitney discs $\widetilde{W}$ with the same boundaries as $W$ but with interiors disjoint from $W$, then the pseudo-isotopy can be smoothly deformed to one with the same $A$, $B$, and $V$, but with $W$ replaced with $\widetilde{W}$. Figure~\ref{F:ABVWWtilde} illustrates this setup. Quinn initially reconfigures things slightly by introducing a new $2$--sphere $\widetilde{A}$ that is parallel to $A$ but is not yet seen as the ascending sphere of a $2$--handle. Then one displaces $\widetilde{W}$ slightly so that $V$ and $\widetilde{W}$ are seen as finger and Whitney discs for the pair $(\widetilde{A},B)$, while $V$ and $W$ remain as finger and Whitney discs for $(A,B)$. This is illustrated in Figure~\ref{F:AAtildeBVWWtilde}.

\begin{figure}[h!]
\begin{center}

\begin{subfigure}{.45\textwidth}
  \centering
  \labellist
\small\hair 2pt
\pinlabel {$A$} [c] at 205 460
\pinlabel {$B$} [c] at 180 260
\pinlabel {$V$} [c] at 308 205
\pinlabel {$W$} [c] at 420 340
\endlabellist
\includegraphics[width=\linewidth]{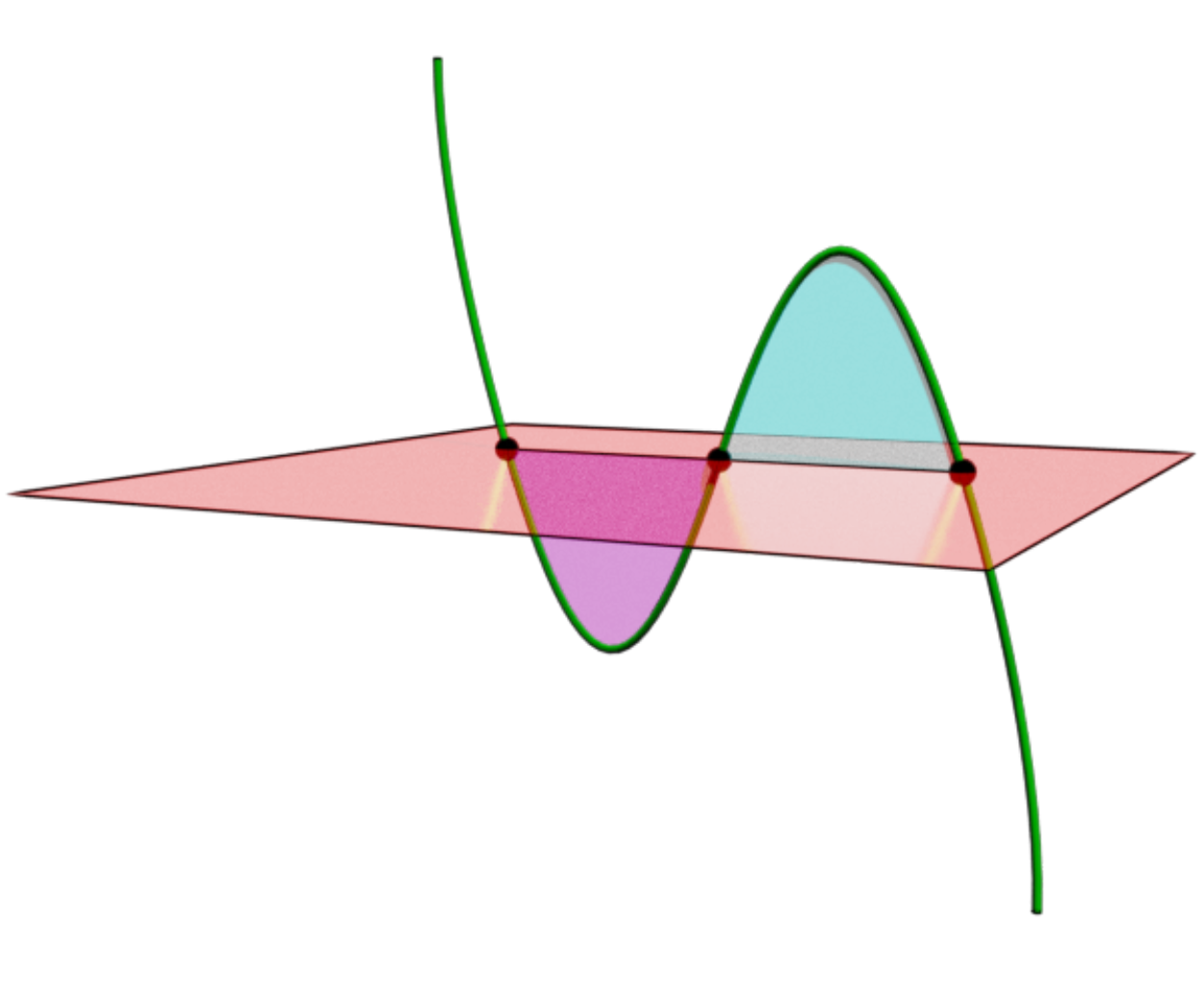}
\end{subfigure}
\begin{subfigure}{.45\textwidth}
  \centering
  \labellist
\small\hair 2pt
\pinlabel {$A$} [c] at 205 460
\pinlabel {$B$} [c] at 180 260
\pinlabel {$V$} [c] at 308 205
\pinlabel {$\widetilde{W}$} [c] at 400 300
\endlabellist
\includegraphics[width=\linewidth]{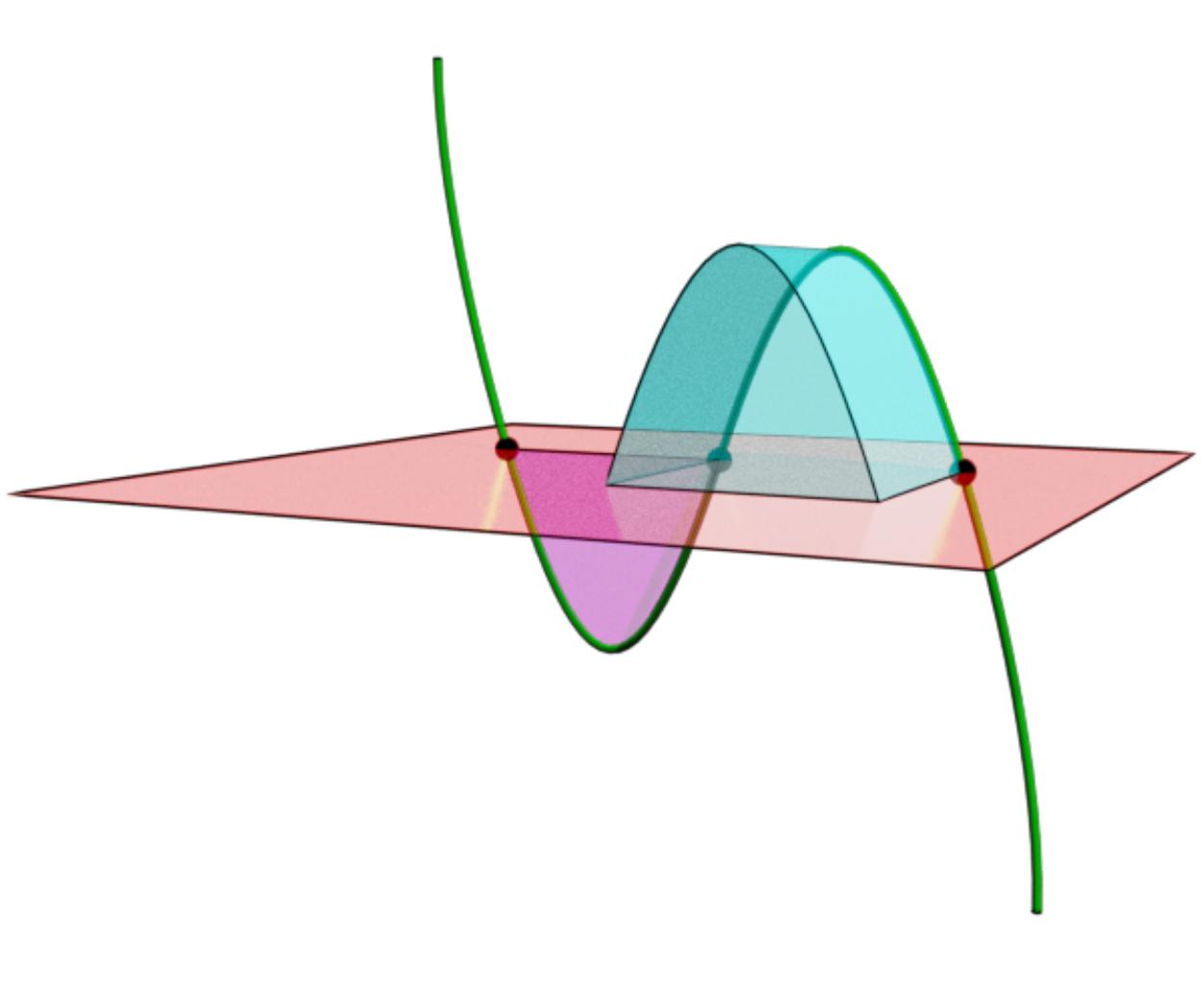}
\end{subfigure}

\caption{On the left: $A$, $B$, $V$, and $W$. On the right: the same $A$, $B$ and $V$ but now we have a different Whitney disc $\widetilde{W}$. We do not assume $\widetilde{W}$ is isotopic to $W$ (if it were, this discussion would be trivial), only that the interiors of $W$ and $\widetilde {W}$ are disjoint. Note that both $W$ and $\widetilde{W}$ may intersect $V$ in their interiors, but this is not indicated in the figure. }
\label{F:ABVWWtilde}
\end{center}
\end{figure}
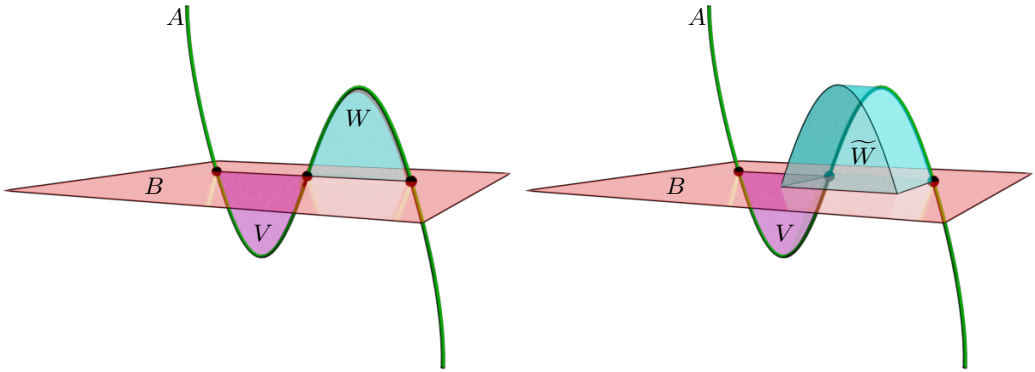

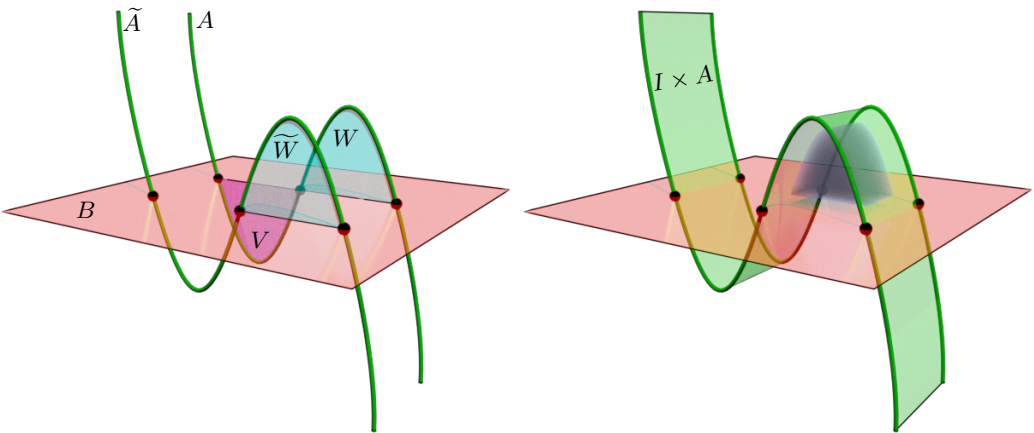
\begin{figure}[h]
\begin{center}

\begin{subfigure}{.45\textwidth}
  \centering
  \labellist
\small\hair 2pt
\pinlabel {$A$} [c] at 240 490
\pinlabel {$B$} [c] at 100 265
\pinlabel {$V$} [c] at 305 230
\pinlabel {$W$} [c] at 405 350
\pinlabel {$\widetilde{A}$} [c] at 154 490
\pinlabel {$\widetilde{W}$} [c] at 335 340
\endlabellist
\includegraphics[width=\linewidth]{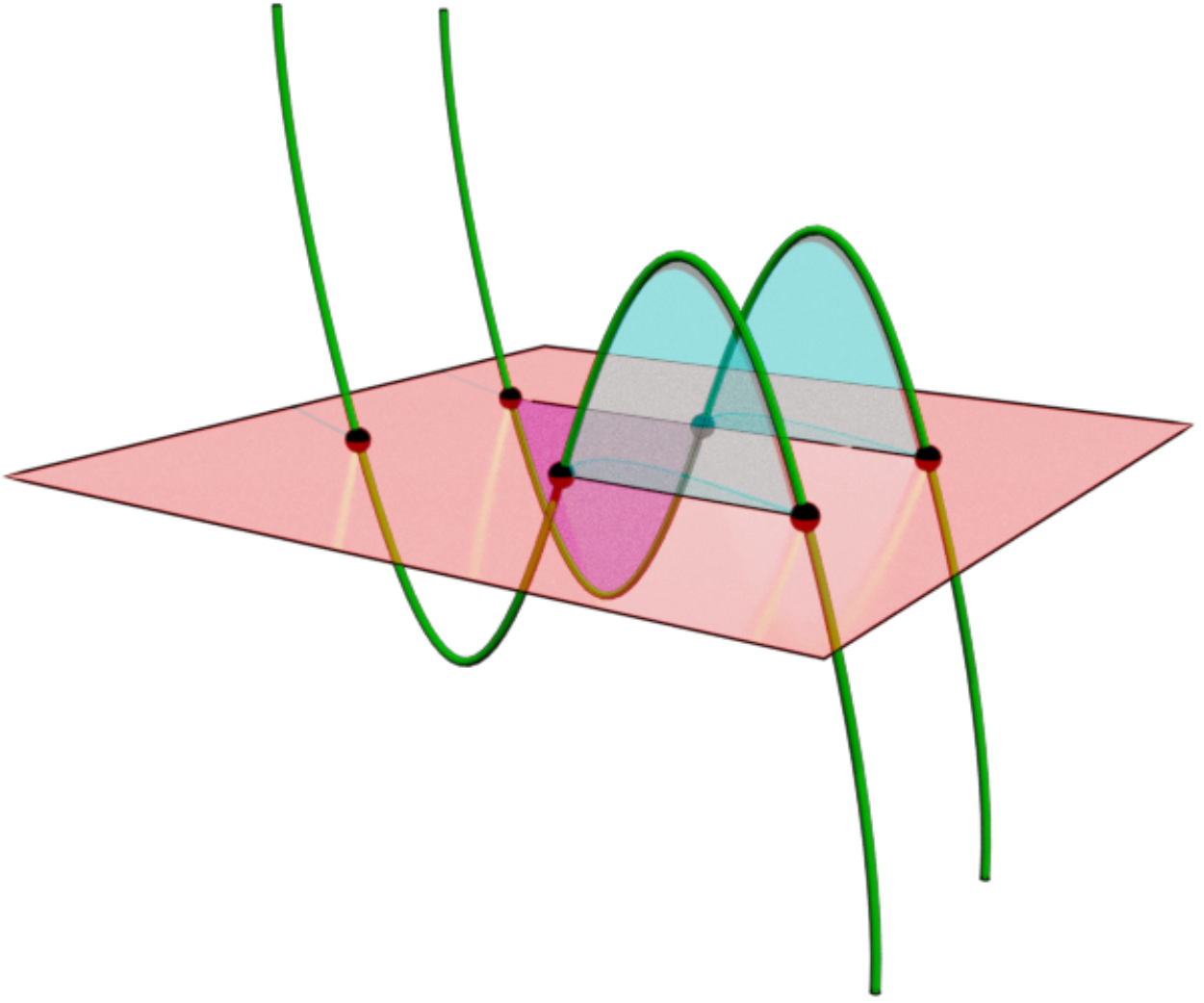}
\end{subfigure}
\begin{subfigure}{.45\textwidth}
  \centering
  \labellist
\small\hair 2pt
\pinlabel {\rotatebox[origin=c]{10}{$I \times A$}} [c] at 188 420
\endlabellist
\includegraphics[width=\linewidth]{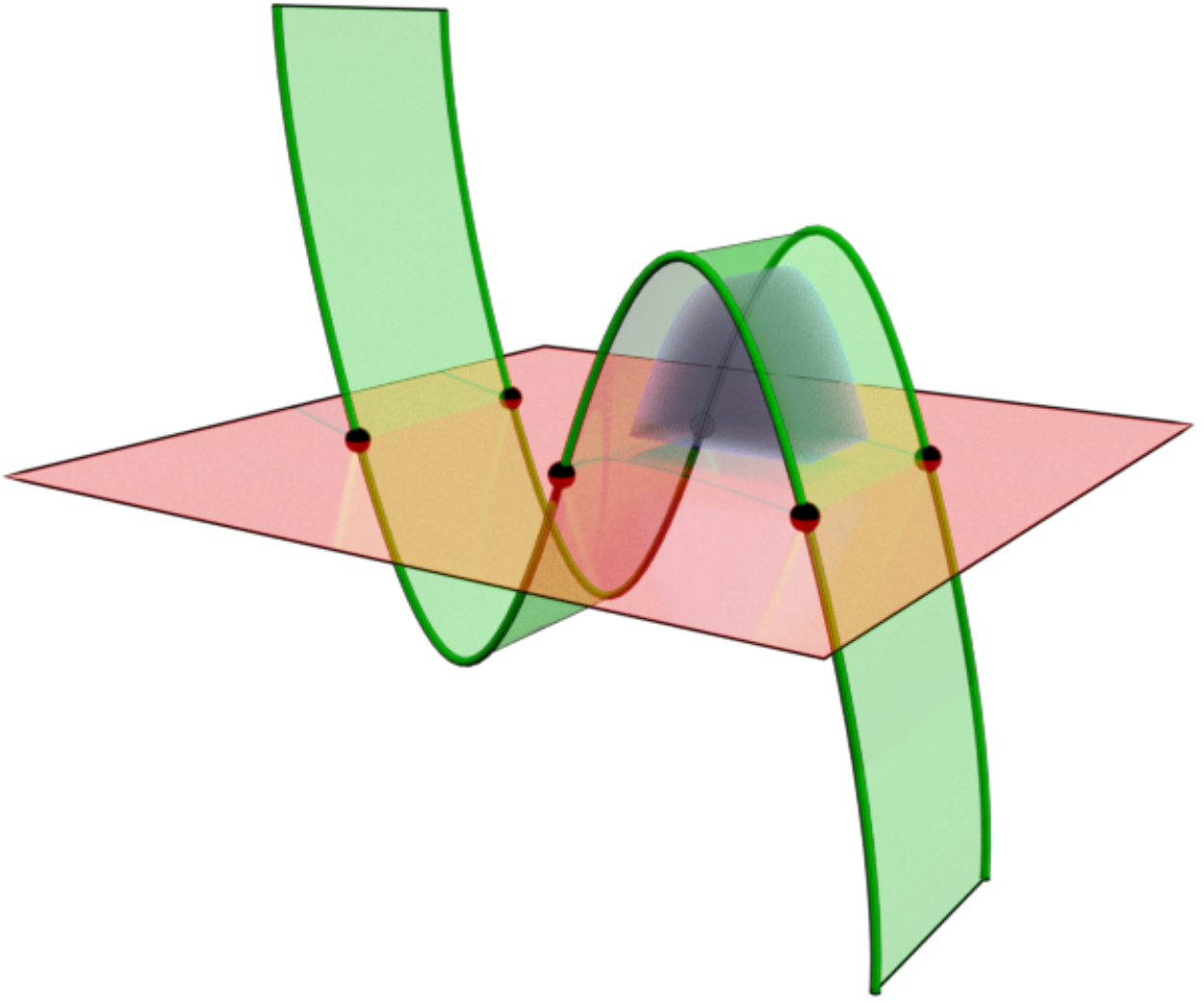}
\end{subfigure}

\caption{On the left: $A$, $B$, $V$, and $W$ together with $\widetilde{A}$, a parallel copy of $A$, and the new version of $\widetilde{W}$. On the right, we highlight the fact that the apparent product $I \times W$ going from $W$ to $\widetilde{W}$ does not necessarily exist in the ambient manifold, and we illustrate this with a grey void where the $I \times W$ would be. However $A$ and $\widetilde{A}$ {\em are} the two ends of an embedding of $I \times A$ which is also shown in the figure on the right. There is also no ``void" on the $V$ side of the picture because the local model is constructed so that a neighbourhood $I \times V$ of $V$ does in fact extend all the way to $\widetilde{A}$. }
\label{F:AAtildeBVWWtilde}
\end{center}
\end{figure}

To describe Quinn's first modification, we introduce the notation $B_V$ to denote the result of modifying $B$ by performing a Whitney move across $V$, and similarly $B_W$; $B_{W,\widetilde{W}}$ denotes the result of a simultaneous Whitney move across $W$ and across $\widetilde{W}$. We initially start with $A$ not moving, and with $B$ moving backwards from $t=1/2$ to $B_V$ and forwards from $t=1/2$ to $B_W$. Although $\widetilde{A}$ is not the ascending sphere of a $2$--handle yet, Quinn's first modification is (a) to make sure that the Whitney move along $V$ is wide enough so that the isotopy from $B$ to $B_V$ eliminates intersections between $\widetilde{A}$ and $B$ at the same time as eliminating the corresponding intersections between $A$ and $B$, and (b) to replace the Whitney move along $W$ with simultaneous Whitney moves along $W$ and $\widetilde{W}$, so that the isotopy from $B$ to $B_{W,\widetilde{W}}$ also eliminates intersections with $\widetilde{A}$ as well as with $A$. This is illustrated in Figure~\ref{F:VMoveWWtildeMove}. The reason that we can make the $V$ move wide but have to do two separate narrow moves on $W$ and $\widetilde{W}$ is that the local model is built so that the product neighbourhood $I \times V$ does stretch from $A$ to $\widetilde{A}$, while $W$ and $\widetilde{W}$ are not assumed to be isotopic, so that there is not a product $I \times W$ doing the same thing on the $W$ side of the picture.

\begin{figure}[h]
\begin{center}

\begin{subfigure}{.45\textwidth}
  \centering
\includegraphics[width=\linewidth]{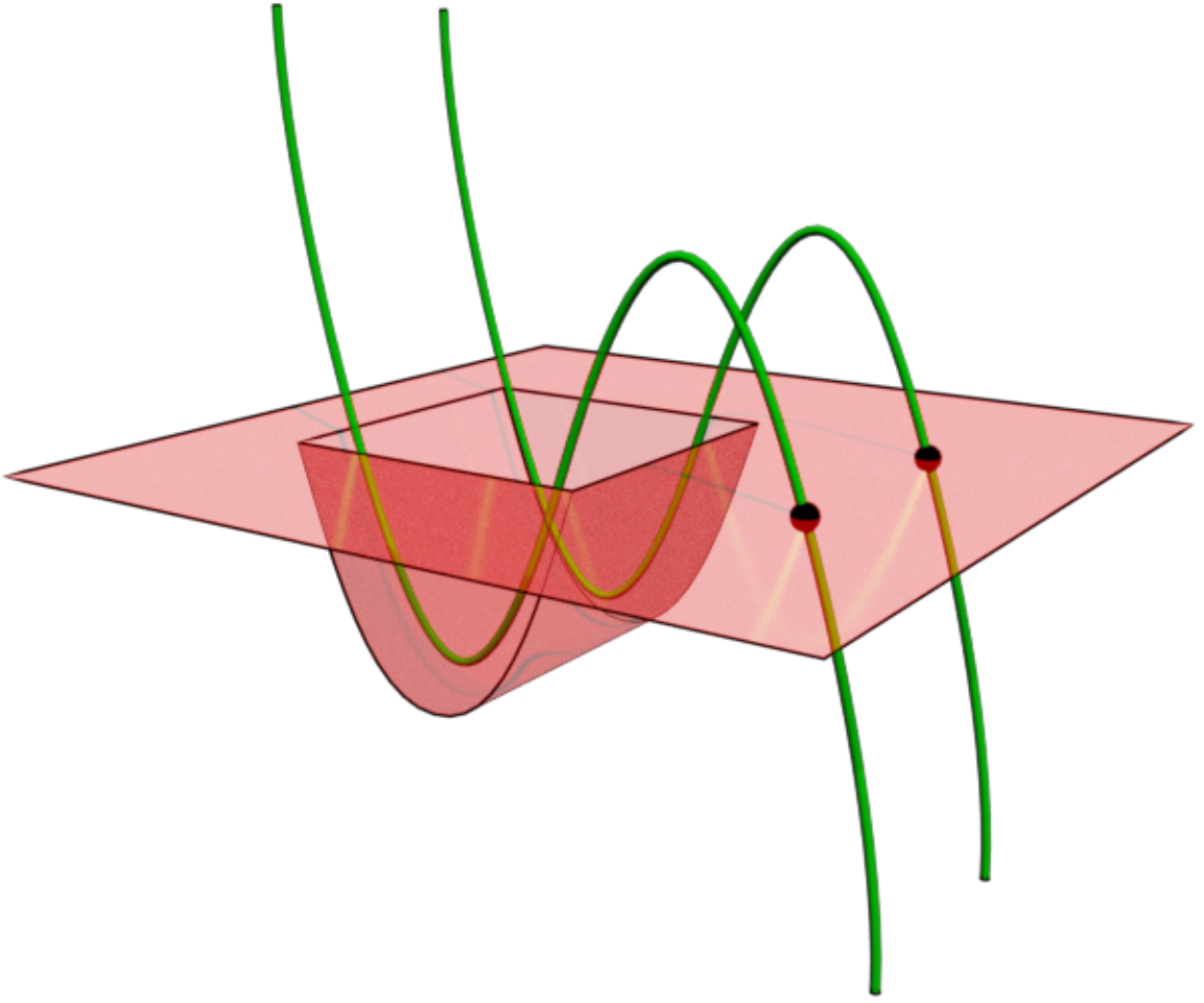}
\end{subfigure}
\begin{subfigure}{.45\textwidth}
  \centering
\includegraphics[width=\linewidth]{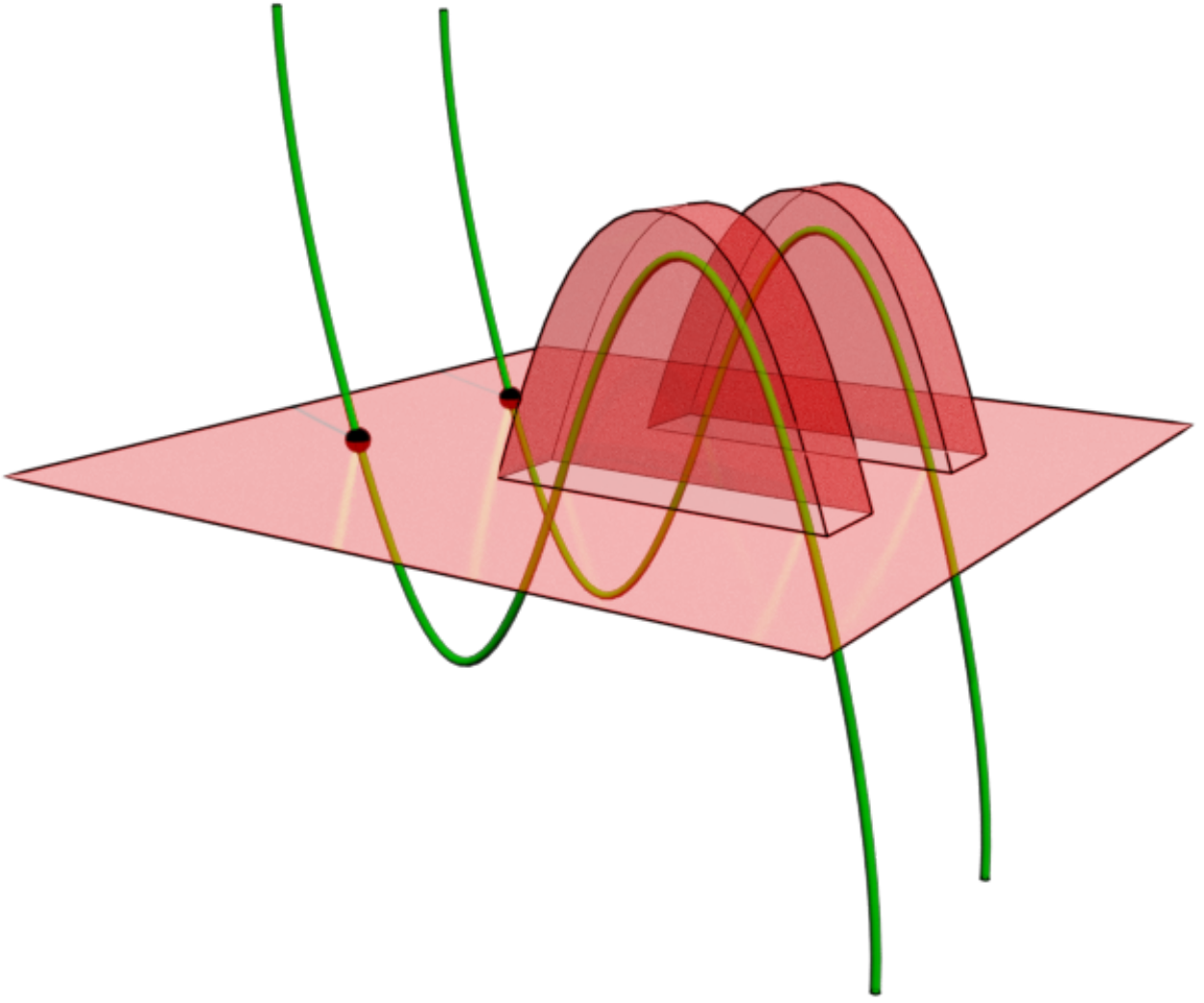}
\end{subfigure}

\caption{The spheres $A$, $\widetilde{A}$, and $B_V$ on the left, and the spheres $A$, $\widetilde{A}$, and $B_{W,\widetilde{W}}$ on the right.  }
\label{F:VMoveWWtildeMove}
\end{center}
\end{figure}

If we follow the $I \times S^2$ bounded by $A$ and $\widetilde{A}$ down to a level below the $2$--handle corresponding to $A$, where the ambient manifold is now $M$, the $I \times S^2$ is surgered to a $D^3$, with boundary equal to $\widetilde{A}$. Turning things upside down, one can use this $3$--ball to create a cancelling $3$--$4$ handle pair, where the boundary of the $3$--ball is the attaching sphere for the $3$--handle and the $3$--ball is half of the attaching sphere for the $4$--handle. Turning things the right way up again, this is a new $1$--$2$ pair, and this birth can happen slightly after time $t_b$,  while the corresponding death happens slightly before time $t_d$. This yields a Cerf graphic as in Figure~\ref{fig:Quinn1223}.
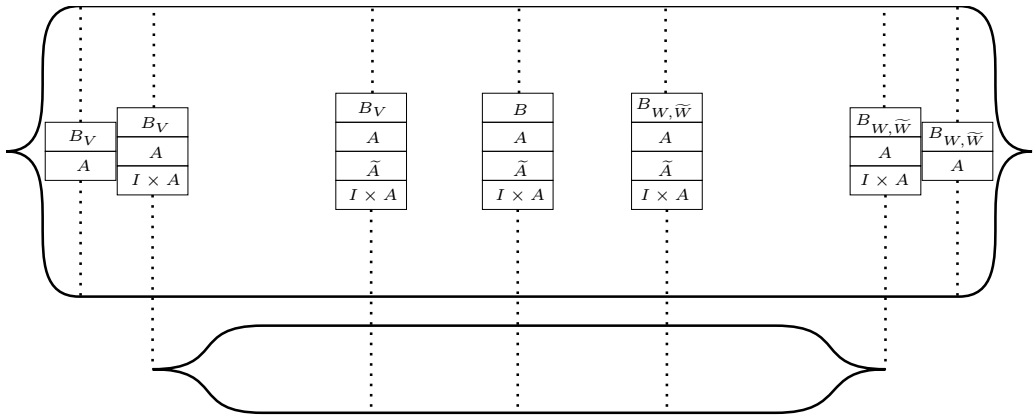
\begin{figure}
  \labellist
  \tiny\hair 2pt
  \pinlabel $B$ [c] at 199 130
  \pinlabel $A$ [c] at 199 119
  \pinlabel {$\widetilde{A}$} [c] at 199 107
  \pinlabel {$I \times A$} [c] at 199 96

  \pinlabel $B_V$ [c] at 142 130
  \pinlabel $A$ [c] at 142 119
  \pinlabel {$\widetilde{A}$} [c] at 142 107
  \pinlabel {$I \times A$} [c] at 142 96

  \pinlabel $B_V$ [c] at 58 124
  \pinlabel $A$ [c] at 58 113
  \pinlabel {$I \times A$} [c] at 58 102

  \pinlabel $B_V$ [c] at 30 119
  \pinlabel $A$ [c] at 30 108

  \pinlabel $B_{W, \widetilde{W}}$ [c] at 255 130
  \pinlabel $A$ [c] at 255 119
  \pinlabel {$\widetilde{A}$} [c] at 255 107
  \pinlabel {$I \times A$} [c] at 255 96

  \pinlabel $B_{W, \widetilde{W}}$ [c] at 340 124
  \pinlabel $A$ [c] at 340 113
  \pinlabel {$I \times A$} [c] at 340 102

  \pinlabel $B_{W, \widetilde{W}}$ [c] at 368 119
  \pinlabel $A$ [c] at 368 108

  \endlabellist

 \centering
 \includegraphics[width=.9\textwidth]{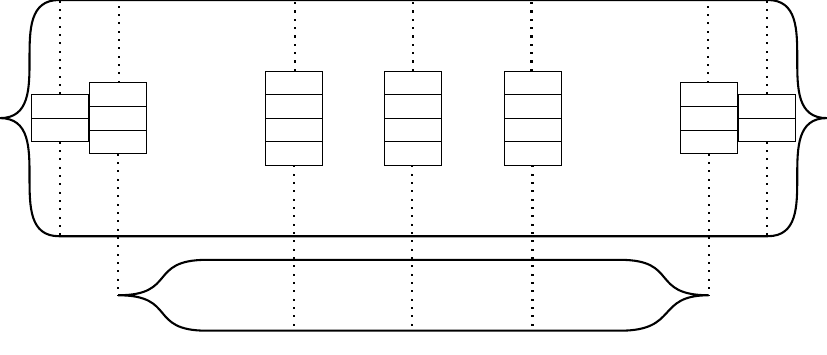}
     \caption{Cerf graphic after introducing a $1$--$2$--pair. The boxes in the interior of the main eye label the ascending and descending spheres of all the critical points. Since the lowest critical point is index $1$, its ascending sphere is $3$--dimensional. However, after rising up past two index $2$ critical points this $S^3$ is punctured twice and appears as an $I \times S^2$, identified with $I \times A$, with boundary equal to $A \amalg \widetilde{A}$. Note that ascending manifolds for birth/death points are also shown.}
    \label{fig:Quinn1223}
\end{figure}

Actually Quinn suggests that the new $1$--$2$ eye should start and end at the same time as the $2$--$3$ eye, rather than being slightly shorter as indicated in our figure. In what follows one might think that this is an important distinction, but we will argue that in fact this does not make a difference. For now, however, if one wishes one can move the cusps of the $1$--$2$ eye left and right as desired. We have not done this partly just because doing so would make it harder to keep track of ascending and descending spheres in the ``middle level''.

Looked at from the original middle level, above the original $2$--handle, the ascending sphere for the new $2$--handle is $\widetilde{A}$ and the ascending $3$--sphere for the new $1$--handle appears as the $I \times S^2$ bounded by $A$ and $\widetilde{A}$. This means that the $A$ $2$--handle can be cancelled with the new $1$--handle from slightly after the birth of the new $1$--$2$ pair to slightly before the death of the new $1$--$2$ pair. Performing this cancellation produces a Cerf graphic as in \cref{figure:Quinn1223swallowtails}.

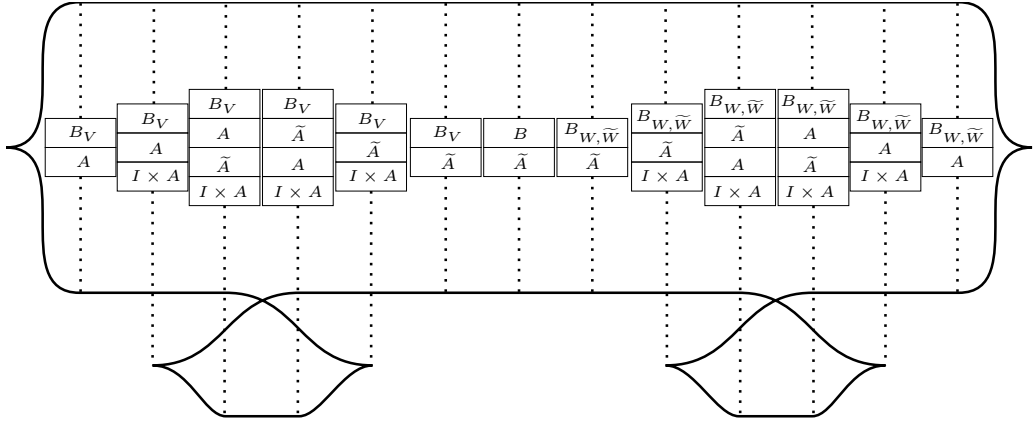
\begin{figure}
  \labellist
  \tiny\hair 2pt
  \pinlabel $B$ [c] at 199 119
  \pinlabel $\widetilde{A}$ [c] at 199 108

  \pinlabel $B_V$ [c] at 171 119
  \pinlabel $\widetilde{A}$ [c] at 171 108

  \pinlabel $B_V$ [c] at 142 124
  \pinlabel $\widetilde{A}$ [c] at 142 113
  \pinlabel {$I \times A$} [c] at 142 102

  \pinlabel $B_V$ [c] at 113 130
  \pinlabel $\widetilde{A}$ [c] at 113 119
  \pinlabel {$A$} [c] at 113 107
  \pinlabel {$I \times A$} [c] at 113 96

  \pinlabel $B_V$ [c] at 84 130
  \pinlabel $A$ [c] at 84 119
  \pinlabel $\widetilde{A}$ [c] at 84 107
  \pinlabel {$I \times A$} [c] at 84 96

  \pinlabel $B_V$ [c] at 58 124
  \pinlabel $A$ [c] at 58 113
  \pinlabel {$I \times A$} [c] at 58 102

  \pinlabel $B_V$ [c] at 30 119
  \pinlabel $A$ [c] at 30 108

  \pinlabel {$B_{W, \widetilde{W}}$} [c] at 227 119
  \pinlabel $\widetilde{A}$ [c] at 227 108

  \pinlabel $B_{W, \widetilde{W}}$ [c] at 255 124
  \pinlabel $\widetilde{A}$ [c] at 255 113
  \pinlabel {$I \times A$} [c] at 255 102

  \pinlabel $B_{W, \widetilde{W}}$ [c] at 283 130
  \pinlabel $\widetilde{A}$ [c] at 283 119
  \pinlabel {$A$} [c] at 283 107
  \pinlabel {$I \times A$} [c] at 283 96

  \pinlabel $B_{W, \widetilde{W}}$ [c] at 311 130
  \pinlabel $A$ [c] at 311 119
  \pinlabel {$\widetilde{A}$} [c] at 311 107
  \pinlabel {$I \times A$} [c] at 311 96

  \pinlabel $B_{W, \widetilde{W}}$ [c] at 340 124
  \pinlabel $A$ [c] at 340 113
  \pinlabel {$I \times A$} [c] at 340 102

  \pinlabel $B_{W, \widetilde{W}}$ [c] at 368 119
  \pinlabel $A$ [c] at 368 108

  \endlabellist

 \centering
 \includegraphics[width=0.9\textwidth]{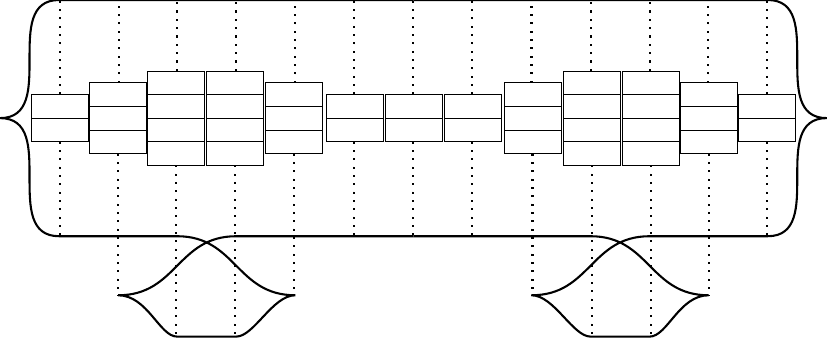}

    \caption{After cancelling the $A$ $2$--handle with the $1$--handle, again showing ascending and descending spheres in the middle level. Note now that between the two swallowtails, the ascending sphere for the index $2$ critical point is now $\widetilde{A}$ instead of $A$, and that the isotopy of $B$ going from the middle-middle level one step to the right, taking $B$ to $B_{W,\widetilde{W}}$, now removes the extra intersection points between $B$ and $\widetilde{A}$.}
    \label{figure:Quinn1223swallowtails}
\end{figure}

Note in this figure that the middle of the Cerf graphic is exactly what we would like it to be: the ascending sphere for the $2$-handle is $\widetilde{A}$, and $B$ moves backwards in time from $t=1/2$ by doing the Whitney move along $V$ so that $B_V$ is in cancelling position with $\widetilde{A}$, while $B$ moves forward in time from $t=1/2$ by doing the Whitney move along $W$ and $\widetilde{W}$, so that $B_{W,\widetilde{W}}$ is in cancelling position with $\widetilde{A}$. If that were all we saw of the Cerf graphic, i.e.\ if the Cerf graphic were cut off before the swallowtails arise, leaving an eye with just the boxes at the middle three time points of \cref{figure:Quinn1223swallowtails} or \cref{fig:SmallerSwallowTails} inside, then we would have the family of handlebodies that we desire. (At the very end, we could do less and less of the Whitney move along $W$ and since $A$ would no longer be in the picture, this would just be a deformation
of the family that does not create extraneous intersections between spheres, and we would be left with $B$ performing just the Whitney move along $\widetilde{W}$.)

Quinn's goal at this point is to cancel the swallowtails while moving them to the beginning and end of the Cerf graphic, thus widening the middle section so that in the end $\widetilde{A}$ is the ascending sphere for the one remaining $2$--handle for the entire time from the birth to the death. The problem is that, at some point, the swallowtails need to shrink, and that, in doing so, the ascending $3$--sphere for the $1$--handle appears in the middle level as a {\em shrinking} $I \times S^2$. In other words, the $I$ factor in the $I \times S^2$ needs to shrink to a point so that immediately after the swallowtail has disappeared, the $I \times S^2$ becomes an $S^2$, the ascending sphere for the single $2$--handle that remains. Figure~\ref{fig:SmallerSwallowTails} shows the same Cerf graphic with smaller swallowtails, on the way towards the swallowtail cancellation, with the emphasis on seeing how the ascending spheres change. Figure~\ref{fig:SmallerSpheres} accompanies Figure~\ref{fig:SmallerSwallowTails} to explain the labelling.

We would like to emphasise the main problem in Quinn's argument. Recall that the two boundary components of $I \times S^2$ in the ascending 3-sphere of the 1-handle are $A$ and $\widetilde{A}$.
Observe that if $W$ and $\widetilde W$ were isotopic (a trivial case in which the replacement criterion has no value), then $A$ and $\widetilde A$ would be  isotopic in the complement of $B_{W, \widetilde W}$.
In the general case that $W$ and $\widetilde W$ are not isotopic, we cannot assume that $A$ and $\widetilde A$ are isotopic in the complement of $B_{W, \widetilde W}$.
So during the shrinking of $I\times S^2$, some 2-spheres of the form $\{ *\}\times S^2$ will  have extra intersections with $B_{W, \widetilde W}$.

\cref{fig:SmallerSwallowTails,fig:SmallerSpheres} illustrate the problem with the scenario where the interval $I$ shrinks to its midpoint. In this case the new boundary 2-spheres of the product $I^*\times S^2=[1/4, 3/4]\times S^2$, denoted $A^*$ and $\widetilde{A^*}$ in the figures, have extra intersections with $B_{W, \widetilde W}$.
Note that $A^*$ and $\widetilde{A^*}$ serve as ascending spheres of the 2-handle for some respective values of the horizontal parameter $t$ in the Cerf graphic after the swallowtails are eliminated; this can be seen in \cref{fig:SmallerSwallowTails}.
The end result of the shrinking, $\{ 1/2\}\times S^2$, has extra intersections with $B_{W, \widetilde W}$ as well.  All these extra intersections give a contradiction with Quinn's claim.

Shrinking $I$ to any other point leads to essentially the same problem.
For example, the interval~$I$ will shrink with one end or the other fixed if, as Quinn proposes, we push the swallowtail all the way to the end as we cancel it.
Shrinking the interval $I$ to its endpoint, say $1$, has an intermediate stage $[1/2,1]\times S^2$. In this case, the two boundary 2-spheres are $A^*:=\{1/2\}\times S^2$ and $\widetilde A=\{1\} \times S^2$, and $A^*$ has extra intersections with $B_{W, \widetilde W}$. Hence analogously to the previous case, after the swallowtail is cancelled, $A^*$ appears as the ascending sphere of the 2-handle for some value of the parameter~$t$. These additional intersections therefore again contradict Quinn's claim.
As a consequence of this discussion we see that, in the course of the swallowtail cancellation, the intersections between $A$ or $\widetilde{A}$ and $B$ which were eliminated by doing Whitney moves across $W$ and $\widetilde{W}$ necessarily reappear.

\begin{figure}
  \labellist
  \tiny\hair 2pt
  \pinlabel $B$ [c] at 199 90
  \pinlabel $\widetilde{A}$ [c] at 199 79

  \pinlabel $B_V$ [c] at 171 90
  \pinlabel $\widetilde{A}$ [c] at 171 79

  \pinlabel $B_V$ [c] at 129 95
  \pinlabel $\widetilde{A^*}$ [c] at 128 84
  \pinlabel {$I^* \!\! \times A$} [c] at 128 73

  \pinlabel $B_V$ [c] at 99 101
  \pinlabel $A^*$ [c] at 99 79
  \pinlabel $\widetilde{A^*}$ [c] at 99 90
  \pinlabel {$I^* \!\! \times A$} [c] at 99 68

  \pinlabel $B_V$ [c] at 72 95
  \pinlabel $A^*$ [c] at 71 84
  \pinlabel {$I^* \!\! \times A$} [c] at 71 73

  \pinlabel $B_V$ [c] at 30 90
  \pinlabel $A$ [c] at 30 79

  \pinlabel {$B_{W, \widetilde{W}}$} [c] at 227 90
  \pinlabel $\widetilde{A}$ [c] at 227 79

  \pinlabel $B_{W, \widetilde{W}}$ [c] at 270 95
  \pinlabel $\widetilde{A^*}$ [c] at 270 84
  \pinlabel {$I^* \!\! \times A$} [c] at 270 73

  \pinlabel $B_{W, \widetilde{W}}$ [c] at 299 101
  \pinlabel $A^*$ [c] at 299 79
  \pinlabel {$\widetilde{A^*}$} [c] at 299 90
  \pinlabel {$I^* \!\! \times A$} [c] at 299 68

  \pinlabel $B_{W, \widetilde{W}}$ [c] at 327 95
  \pinlabel $A^*$ [c] at 327 84
  \pinlabel {$I^* \!\! \times A^*$} [c] at 327 73

  \pinlabel $B_{W, \widetilde{W}}$ [c] at 368 90
  \pinlabel $A$ [c] at 368 79

  \endlabellist

 \centering
 \includegraphics[width=0.9\textwidth]{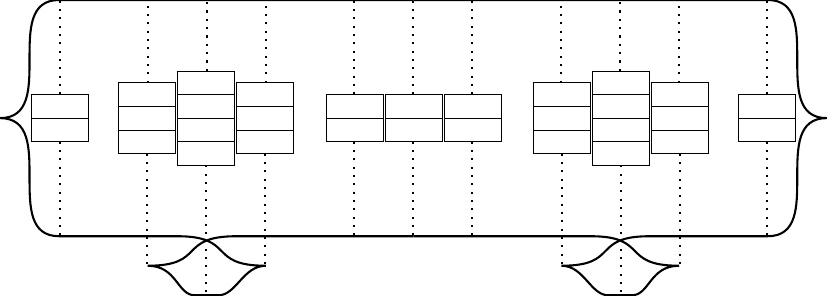}

 \caption{The problem appears here when we start to shrink the swallowtails. If $I = [0,1]$, then $I^*$ is smaller, e.g.\ $I^* = [1/4,3/4]$. The new ascending spheres $A^*$ and $\widetilde{A^*}$ are $\{1/4\} \times A$ and $\{3/4\} \times A$, if $A$ was originally identified with $\{0\} \times A$ and $\widetilde{A}$ was originally identified with $\{1\} \times A$. These spheres are illustrated in Figure~\ref{fig:SmallerSpheres}.}
 \label{fig:SmallerSwallowTails}
\end{figure}

\begin{figure}
    \centering

 \includegraphics[width=.4\textwidth]{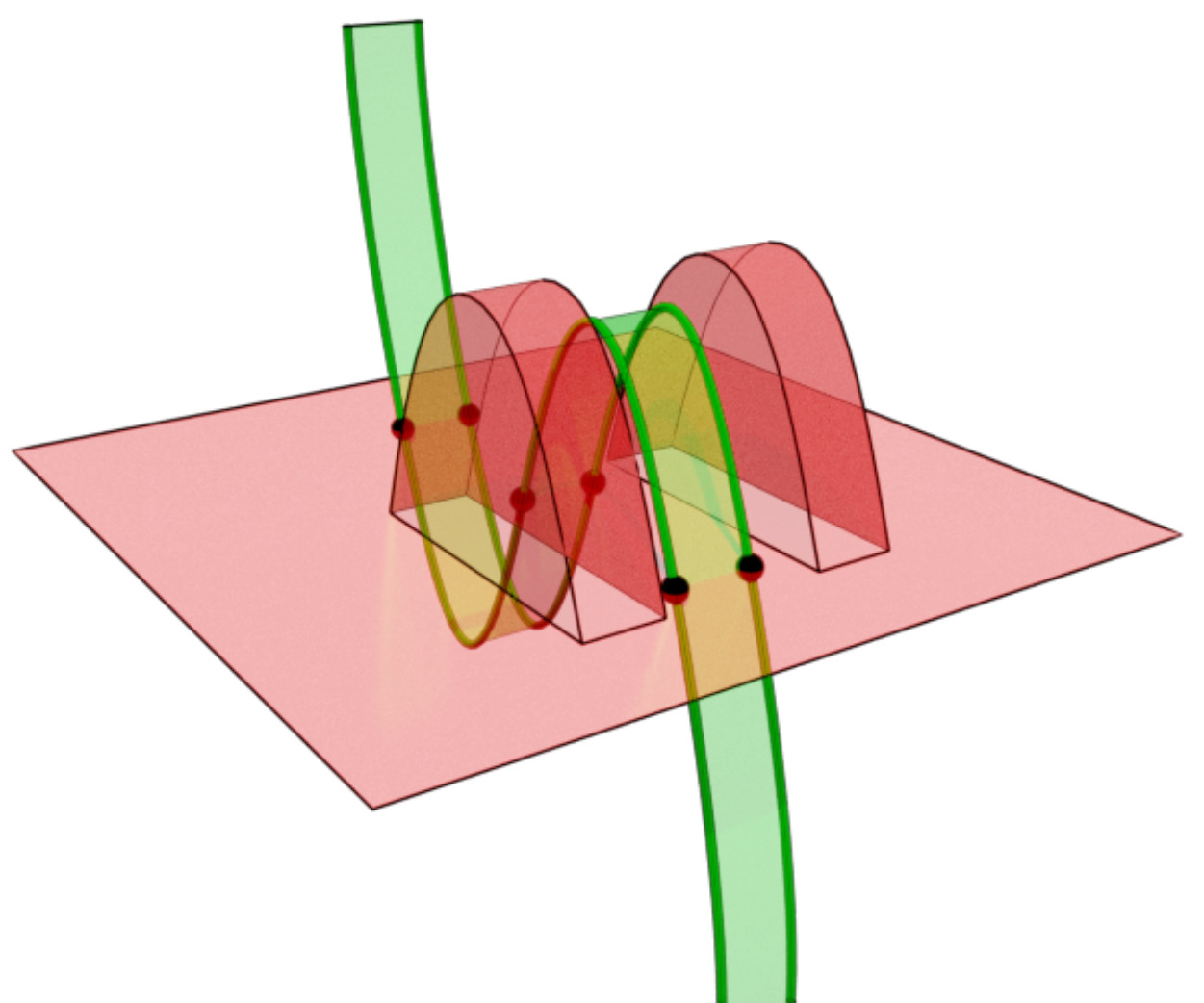}

 \caption{ Illustration to accompany Figure~\ref{fig:SmallerSwallowTails}, showing the ascending manifold $I^* \times A$ for the $1$--handle and the ascending spheres $A^*$ and $\widetilde{A^*}$ for the $2$--handles. Note that both $A^*$ and $\widetilde{A^*}$ are back to having three intersection points with $B_{W,\widetilde{W}}$.}
 \label{fig:SmallerSpheres}
\end{figure}

Now we show that in fact the end result of Quinn's method is nothing more than factorization. For the sake of concreteness in the exposition below, we consider the scenario where $I$ shrinks to its midpoint, but the same conclusion can be made regardless of how the interval shrinks. After completely cancelling the swallowtail, we are left with the ascending sphere for the $2$--handle at the time of the cancellation being a sphere which we now call $\overline{A}$, equal to $\{1/2\} \times A$ in the original $I \times A$. The resulting Cerf graphic with ascending and descending spheres is shown in Figure~\ref{fig:QuinnCancelled}. The key point to note is that now the ascending $A$ spheres are moving in time, whereas originally they were fixed and only the $B$ spheres moved. Furthermore, the sphere $\overline{A}$ moves forward in time to $A$ by a Whitney move along a new Whitney disc which we call $W^*$, and moves backwards in time to $\widetilde{A}$ by a Whitney move along a Whitney disc which we call $\widetilde{A^*}$. These Whitney (finger) moves are indicated by labelled arrows in Figure~\ref{fig:QuinnCancelled}, and the discs themselves are illustrated in Figure~\ref{fig:NewWhitneyDisks}.

\begin{figure}
  \labellist
  \tiny\hair 2pt
  \pinlabel $B$ [c] at 199 90
  \pinlabel $\widetilde{A}$ [c] at 199 79

  \pinlabel $B_V$ [c] at 143 90
  \pinlabel $\widetilde{A}$ [c] at 143 79

  \pinlabel $B_V$ [c] at 85 90
  \pinlabel $\overline{A}$ [c] at 85 79

  \pinlabel $V$ [c] at 171 110
  \pinlabel {$\widetilde{W},\simeq$} [c] at 228 110

  \pinlabel $\simeq$ [c] at 59 89
  \pinlabel $\simeq$ [c] at 113 89

  \pinlabel $B_V$ [c] at 30 90
  \pinlabel $A$ [c] at 30 79

  \pinlabel {$B_{W, \widetilde{W}}$} [c] at 255 90
  \pinlabel $\widetilde{A}$ [c] at 255 79

  \pinlabel $B_{W, \widetilde{W}}$ [c] at 312 90
  \pinlabel $\overline{A}$ [c] at 312 79

  \pinlabel $B_{W, \widetilde{W}}$ [c] at 368 90
  \pinlabel $A$ [c] at 368 79

  \pinlabel $\widetilde{W^*}$ [c] at 288 59
  \pinlabel {$W^*$} [c] at 340 60

  \endlabellist

 \centering
 \includegraphics[width=.9\textwidth]{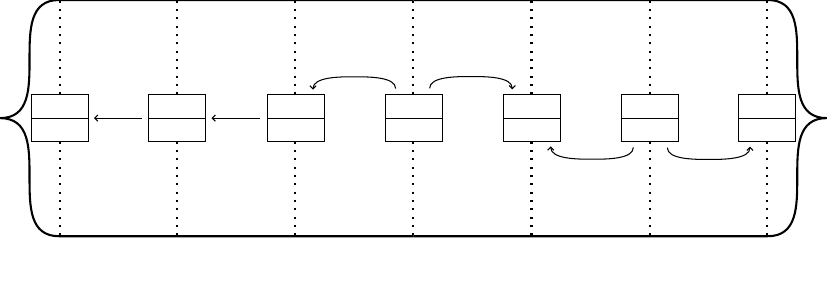}

     \caption{After cancelling the swallowtails. The arrows indicate finger or Whitney moves. The arrow labelled ``$\widetilde{W},\simeq$'' indicating a Whitney move across $\widetilde{W}$ and an isotopy, the isotopy being the Whitney move across $W$. The new finger/Whitney discs $\widetilde{W^*}$ and $W^*$ are illustrated in the accompanying Figure~\ref{fig:NewWhitneyDisks}.}
    \label{fig:QuinnCancelled}
\end{figure}

\begin{figure}
 \centering

 \includegraphics[width=.4\textwidth]{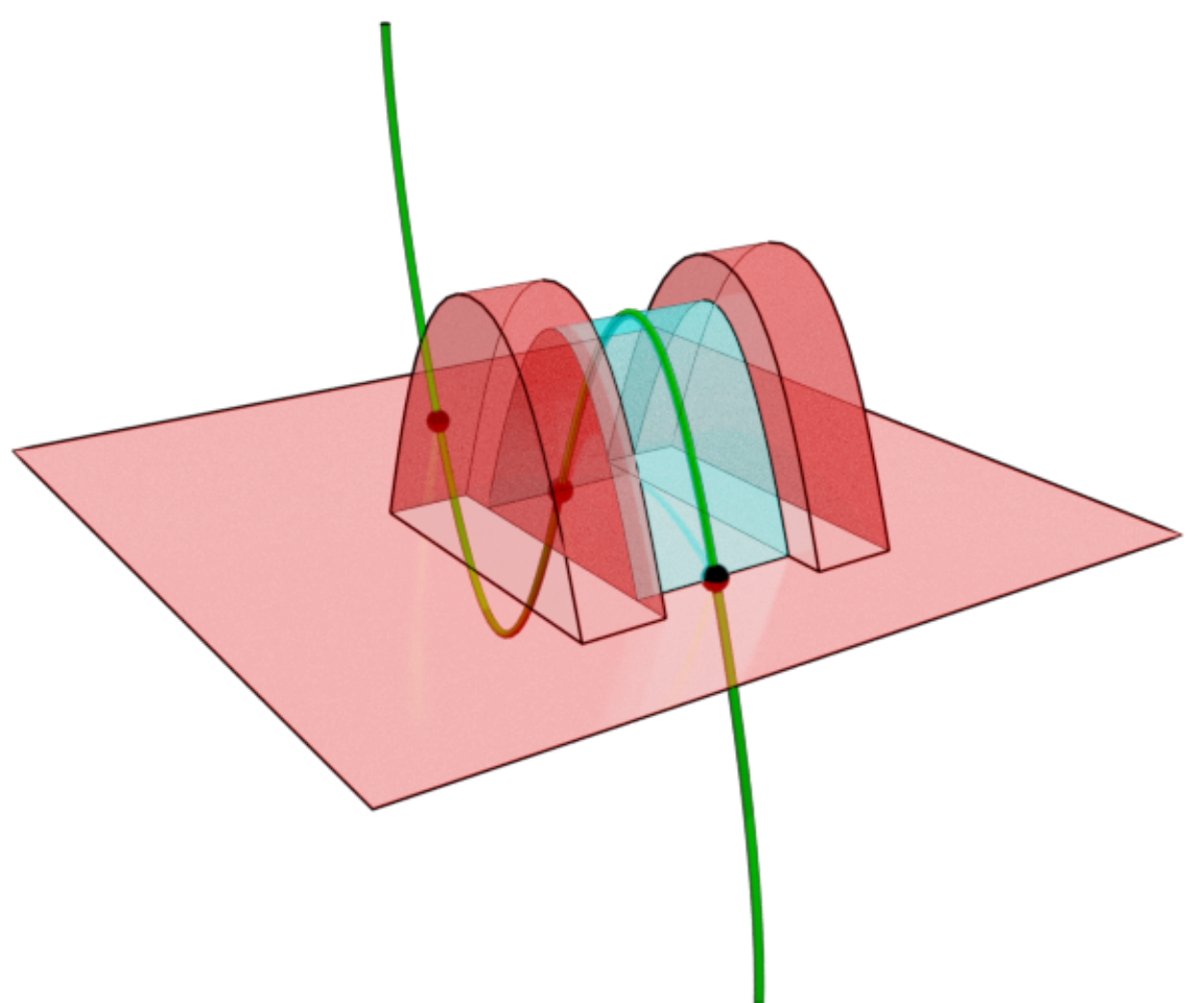}

 \caption{The spheres $\overline{A}$ and $B_{W,\widetilde{W}}$, and the Whitney discs $W^*$ and $\widetilde{W^*}$. These Whitney discs are built from a subset of the strip labelled $I \times A$ in \cref{F:AAtildeBVWWtilde}.}
 \label{fig:NewWhitneyDisks}
\end{figure}

The final observation is that the $4$--tuple $(\overline{A},B_{W,\widetilde{W}},W^*,\widetilde{W^*})$ is in fact isotopic to the original $4$--tuple $(A,B,W,\widetilde{W})$ as seen in Figure~\ref{F:ABVWWtilde}. Thus after an isotopy in the middle middle level, the family illustrated in Figure~\ref{fig:QuinnCancelled} becomes a family in which, starting at the middle and working to the right, $B$ does a Whitney move across $\widetilde{W}$, then undoes that with a finger move back across $\widetilde{W}$, and then does the original Whitney move across $W$.

In summary, the problem in the proof given by Quinn in  \cite{Quinn:isotopy}*{Section~4.5} is in the very last sentence of the proof. One can remove the right hand swallowtail from the Cerf graphic, but the subtlety is that this necessitates an isotopy between the spheres $A$ and $\wt{A}$.  The resulting isotopy intersects $B$ in general, introducing new finger and Whitney moves between $\wt{A}$ and $B$ that were missed by Quinn. The outcome is a collection of finger-Whitney moves with discs $\wt{W}\cdot W$.  In other words, it is exactly the same as the outcome of a factorisation inserting $\ul{\wt{W}}\cdot \wt{W}$.  But a factorisation is technically much simpler, and one is still left with a problem to solve.  So in our proofs in the earlier sections of this paper we appeal to factorisation instead of the method of proof of \cite{Quinn:isotopy}*{Section~4.5}.

An isotopy between $A$ and $\wt{A}$ is destined to create intersections with $B$ if the union of Whitney discs $\wt{W} \cup W$, pushed slightly into the complement of $A$ and $B$, is a homotopically essential 2-sphere in the complement of $A\cup B$.

This leads to the extremely interesting open question of whether the replacement criterion holds in the smooth category, in the simply connected case, especially if one is permitted the additional assumption that the discs one wishes to switch are homotopic rel.\ boundary in the middle-middle level $M \#^k S^2 \times S^2$ (but not necessarily in the complement of $A \cup B$).

\begin{problem}\label{problem:DRC}
   Consider a pseudo-isotopy of a smooth, 1-connected 4-manifold $M$, with associated 1-parameter family having data $(A,B,V,W)$ that satisfies Quinn's arc condition.  Let $\wt{W}$ be a collection of Whitney discs in one to one correspondence with the discs in $W$, that pairwise have the same boundary as $W$, and induce compatible framings.  Suppose that the interiors of $W$ and $\wt{W}$ are disjoint, and that $W_k \cup \wt{W}_k$ is trivial in $\pi_2(M \#^n S^2 \times S^2)$, for all $k$. Is there a deformation of the pseudo-isotopy replacing the family $W$ with the family $\wt{W}$?
\end{problem}

\begin{remark}
    The key inductive step in our proof of \cref{thm:PI-stable-I}, at the end of \cref{section:fixing-the-proof}, involves replacing one system of finger discs $V$ with another $\wh{V}$. To do this we need that the replacement discs $\wh{V}$ have framed embedded geometrically dual spheres and then we need to stabilise the $4$-manifold. Switching the r\^{o}les of $V$ and $W$, this can be thought of as a stable version of the disc replacement criterion: the disc replacement criterion is {\em stably} true provided the replacement discs have dual spheres.
\end{remark}

\subsection{An interesting diffeomorphism of $S^4$}\label{subsec:interesting-diffeo-S4}

We illustrate a potential application of Problem~\ref{problem:DRC} in a particular example. Start with a trivial pseudo-isotopy of $S^4$ whose Cerf graphic is empty. Deform this pseudo-isotopy by creating a single $2, 3$-handle eye   with no finger or Whitney moves.  Deform this family of generalised Morse functions further to one where the spheres $A, B$ undergo a single finger move and a single Whitney move. The finger and Whitney discs are standard and satisfy Quinn's arc condition, as shown in the $3$-dimensional slice in Figure~\ref{fig: a special case of Whitney}.

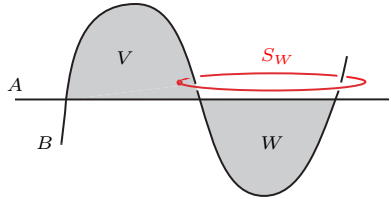
\begin{figure}[h]
\centering
\includegraphics[height=2.7cm]{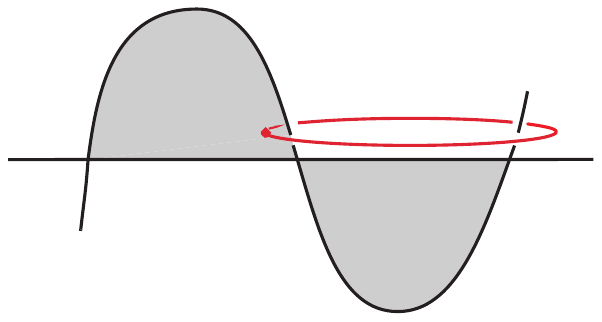}
{\scriptsize
\put(-145,42){$A$}
\put(-134, 20){$B$}
\put(-104,52){$V$}
\put(-50,20){$W$}
\put(-50,53){{\color{red}$S_W$}}
}
\caption{The data of finger and Whitney discs in the middle-middle level determining a potentially nontrivial pseudo-isotopy of $S^4$. }
\label{fig: a special case of Whitney}
\end{figure}

Recall the construction of Whitney spheres from \cref{subsection:whitney spheres}. Consider the Whitney sphere $S_W$; it is disjoint from $W$ and intersects $V$ in a single point. Consider a disc $\widetilde W$ whose boundary is identical to $\partial W$ and whose interior is a slight displacement of that of $W$, tubed into $S_W$. Now we consider a new pseudo-isotopy determined by the pair $(V, \widetilde W)$.  It gives rise to a self-diffeomorphism $f$ of $S^4$. Since $V$ intersects $\widetilde W$ in a point, there is no immediate way to trivialise this pseudo-isotopy.

\begin{conjecture}\label{conjecture}
 The diffeomorphism $f \colon S^4 \to S^4$ is not smoothly isotopic to the identity.
 \end{conjecture}

Note that the interiors of $W$ and $\widetilde W$ are disjoint and $W \cup \wt{W} = 0 \in \pi_2(S^2 \times S^2)$. Thus if the answer to \cref{problem:DRC} is affirmative, then $f$ would be smoothly isotopic to the identity, and \cref{conjecture} would be false.

\def\MR#1{}
\bibliography{bibliography}

\end{document}